\documentclass[11pt, letterpaper, leqno]{article}
\usepackage{amsmath,amsthm,amscd,amssymb,amsbsy,cite,amstext,mathtools}
\usepackage{pgf}
\usepackage{color}
\usepackage{enumerate}
\usepackage{mathrsfs}
\usepackage[symbol]{footmisc}
\usepackage[colorlinks=true, pdfstartview=FitV, linkcolor=black, citecolor=orange, urlcolor=blue]{hyperref}

\usepackage{amsfonts}
\usepackage[T1]{fontenc}
\usepackage{euler}

\theoremstyle{plain}
\newtheorem{theorem}{Theorem}
\newtheorem{lemma}{Lemma}[section]

\newtheorem{claim}{Claim}[section]
\newtheorem{problem}{Problem}

\theoremstyle{remark}
\newtheorem{remark}{Remark}[section]
\newtheorem*{acknowledgment}{Acknowledgment}

\theoremstyle{definition}
\newtheorem{algorithm}{Algorithm}
\newtheorem*{algorithm*}{Algorithm}
\newtheorem{definition}{Definition}[section]

\numberwithin{equation}{section}
\newcommand{\eqindent}{\displayindent0pt\displaywidth\textwidth}

\newcommand{\namedtheorem}[3][2]{
	\theoremstyle{plain}
	\newtheorem*{#1}{#1}
	\begin{#1}[#3]
		#2
	\end{#1}
}

\makeatletter
\newenvironment{subtheorem}[1]{%
	\def\subtheoremcounter{#1}%
	\refstepcounter{#1}%
	\protected@edef\theparentnumber{\csname the#1\endcsname}%
	\setcounter{parentnumber}{\value{#1}}%
	\setcounter{#1}{0}%
	\expandafter\def\csname the#1\endcsname{\theparentnumber.\Alph{#1}}%
	\expandafter\def\csname theH#1\endcsname{theorem.\theparentnumber.\Alph{#1}}%
	\unskip\ignorespaces
}{%
	\setcounter{\subtheoremcounter}{\value{parentnumber}}%
	\ignorespacesafterend
}
\makeatother
\newcounter{parentnumber}

\newcommand{\supp}[1]{\mathrm{supp}\left( {#1} \right)}
\newcommand{\norm}[1]{\| {#1}\| }
\newcommand{\diam}{\mathrm{diam}}
\newcommand{\set}[1]{\left\{#1\right\}}
\newcommand{\void}{\varnothing}
\newcommand{\abs}[1]{\left|#1\right|}
\newcommand{\dist}[2]{\mathrm{dist}\left({#1}, {#2}\right)}

\newcommand{\jet}{\mathscr{J}}
\newcommand{\Q}{\mathcal{Q}}
\newcommand{\R}{\mathbb{R}}
\newcommand{\brac}[1]{\left(#1\right)}
\renewcommand{\d}{\partial}
\newcommand{\Rn}{\mathbb{R}^n}
\newcommand{\grad}{\nabla}
\newcommand{\qt}[1]{``\hspace{1sp}#1\,''}

\newcommand{\E}{\mathcal{E}}
\newcommand{\ct}{C^2}
\newcommand{\ctp}{C^2_+}
\newcommand{\rt}{\R^2}
\renewcommand{\P}{\mathcal{P}}
\newcommand{\pos}{[0,\infty)}

\newcommand{\for}{\enskip\text{for}\enskip}
\newcommand{\ddtm}{\frac{d^m}{dt^m}}
\newcommand{\dq}{\delta_Q}
\newcommand{\Ls}{\Lambda^\sharp}

\begin{document}

	\title{$ {C}^2(\mathbb{R}^2) $ Nonnegative Extension by Bounded-depth Operators}
	\author{Fushuai Jiang \and Garving~K. Luli}
	\maketitle

	\begin{abstract}
		In this paper, we prove the existence of a nonnegative parameter-dependent (nonlinear) $ C^2(\mathbb{R}^2) $ extension operator with bounded depth.
		
	\end{abstract}

	\section{Introduction}

	For nonnegative integers $ m,n $, we write $ C^m(\Rn) $ to denote the Banach space of $ m $-times continuously differentiable real-valued functions such that the following norm is finite
	\begin{equation*}
	\norm{F}_{C^m(\Rn)} := \sup\limits_{x \in \Rn}\brac{  \sum\limits_{\abs{\alpha} \leq m} \abs{ \partial^\alpha F(x) }^2   }^{1/2}\,.
	\end{equation*}
	If $ S $ is a finite set, we write $ \#(S) $ to denote the number of elements in $ S $. We use $ C $ to denote constants that depend only on $ m $ and $ n $.

	Building on \cite{JL20}, in this paper, we prove the following theorem, which was announced in \cite{JL20}.
	
	\begin{theorem}\label{thm.bd}
		Let $ E \subset \rt $ be a finite set. There exist (universal) constants $C, D$, and a map $ \E :   \ctp(E) \times \pos \to \ctp(\rt)  $ such that the following hold. 
		
		\begin{enumerate}[(A)]
			\item Let $ M \geq 0 $. Then for all $ f \in \ctp(E) $ with $ \norm{f}_{\ctp(E)} \leq M $, we have $ \E(f,M) = f $ on $ E $ and 
			$ \norm{\E(f,M)}_{\ct(\rt)} \leq CM $.
			
			\item For each $ x \in \rt $, there exists a set $ S(x) \subset E $ with $ \# (S(x)) \leq D $ such that for all $ M \geq 0 $ and $ f, g \in \ctp(E) $ with $ \norm{f}_{\ctp(E)}, \norm{g}_{\ctp(E)} \leq M $ and $ f|_{S(x)} = g|_{S(x)} $, we have
			\begin{equation*}
			\d^\alpha \E(f,M)(x) = \d^\alpha \E(g,M)(x)
			\for
			\abs{\alpha} \leq 2\,.
			\end{equation*}
			
		\end{enumerate}
	
	\end{theorem}

A few remarks on Theorem \ref{thm.bd} are in order. First of all, in \cite{JL20}, we showed that the extension operator $\E$ is not linear. The constant $D$ appearing in Theorem \ref{thm.bd} is called the \underline{depth} of the extension operator $\E$. This generalizes the notion of the depth of a {\em linear} extension operator first studied by C. Fefferman in \cite{F05-L,F07-St} (for further discussion on the depth of linear extension operators see also G.K. Luli \cite{L10}). The depth of an extension operator (both linear and nonlinear) measures the computational complexity of the extension. The existence of a linear extension operator of bounded depth is one of the main ingredients for the Fefferman-Klartag \cite{FK09-Data-1, FK09-Data-2}  and Fefferman  \cite{F09-Data-3} algorithms for solving the interpolation problems without the nonnegative constraints; the algorithms in \cite{F09-Data-3,FK09-Data-1, FK09-Data-2} are essentially the best possible. 

Armed with Theorem \ref{thm.bd}, in a spirit similar to \cite{F09-Data-3,FK09-Data-1, FK09-Data-2}, we will provide algorithms for solving the following two problems, in the case when $m=n=2$. We remark in passing that Problem 2 is an open problem posed in the book \cite{FI20-book}. Here we announce the algorithms and in a sequel paper \cite{JL20-Alg} we will provide the detailed explanations. As far as we know, there has been no prior work on the two problems. 
	
	\begin{problem}\label{prob.norm}
		Let $ E \subset \Rn $ be a finite set. Let $ f : E \to \pos $. Compute the order of magnitude of 
		\begin{equation}
		\norm{f}_{C^m_+(E)} := \inf \set{ \norm{F}_{C^m(\Rn)} : F|_E = f\text{ and } F \geq 0 }\,. \label{eq.trace}
		\end{equation}
	\end{problem}
	
	\begin{problem}\label{prob.interpolant}
		Let $ E \subset \Rn $ be a finite set. Let $ f : E \to \pos $. Compute a nonnegative function $ F \in C^m(\Rn) $ such that $ F|_E = f $ and $ \norm{F}_{C^m(\Rn)} \leq C\norm{f}_{C^m_+(E)} $. 
	\end{problem}

	By \qt{order of magnitude} we mean the following: Two quantities $ M $ and $ \tilde{M} $ determined by $ E,f,m,n $ are said to have the \underline{same order of magnitude} provided that $ C^{-1}M \leq \tilde{M} \leq CM $, with $ C $ depending only on $ m $ and $ n $. To compute the order of magnitude of $ \tilde{M} $ is to compute a number $ M $ such that $ M $ and $ \tilde{M} $ have the same order of magnitude. 
	
	By \qt{computing a function $F$} from $ (E,f) $, we mean the following: After processing the input $ (E,f) $, we are able to accept query consisting of a point $ x \in \Rn $, and produce a list of numbers $ (f_\alpha(x) : \abs{\alpha} \leq m) $. The algorithm \qt{computes the function $ F $} if for each $ x \in \Rn $, we have $ \d^\alpha F(x) =f_\alpha(x) $ for $ \abs{\alpha} \leq m $. 
	
	We call the quantity in \eqref{eq.trace} the \qt{trace norm} of $ f $ (restricted to $ E $).

	We assume that our algorithms run on an idealized computer with standard von Neumann architecture able to process exact real numbers. We refer the readers to \cite{FK09-Data-2} for a discussion on finite-precision computing. 
	
	%We refer the readers to \cite{JL20} and references therein for the history of Problems \ref{prob.norm} and \ref{prob.interpolant}.

	Recall the definition of the trace norm in \eqref{eq.trace}. 
	Our solution (see \cite{JL20-Alg}) to Problem \ref{prob.norm} consists of the following.

	\begin{algorithm}\label{alg.norm}
		\begin{center}
			Nonnegative $ C^2(\R^2) $ Interpolation Algorithm - Trace Norm
		\end{center}
		\begin{itemize}
			\item[] \textbf{DATA:} $ E \subset \R^2 $ finite with $ \#(E) = N $.
			\item[] \textbf{QUERY:} $f: E \to 
			\pos $.
			\item[] \textbf{RESULT}: The order of magnitude of $ \norm{f}_{C^2_+(E)} $. More precisely, the algorithm outputs a number $ M \geq 0 $ such that both of the following hold.
			\begin{itemize}
				\item We guarantee the existence of a function $ F \in C^2_+(\R^2) $ such that $ F|_E = f $ and $ \norm{F}_{C^2(\R^2)} \leq CM $.
				\item We guarantee there exists no $ F \in C^2_+(\R^2) $ with norm at most $ M $ satisfying $ F|_E = f $.
			\end{itemize}
			
			\item[] \textbf{COMPLEXITY:} 
			\begin{itemize}
				\item Preprocessing $ E $: at most $ CN\log N $ operations and $ CN $ storage.
				\item Answer query: at most $ CN $ operations.
			\end{itemize} 
		\end{itemize}
		
	\end{algorithm}

	Our solution (see \cite{JL20-Alg}) to Problem \ref{prob.interpolant} consists of the following.

	\begin{algorithm}\label{alg.interpolant}
		\begin{center}
			Nonnegative $ C^2(\R^2) $ Interpolation Algorithm - Interpolant
		\end{center}
		\begin{itemize}
			\item[] \textbf{DATA:} $ E \subset \R^2 $ finite with $ \#(E) = N $. $ f: E \to \pos $. $ M \geq 0 $. 
			\item[] \textbf{ORACLE:} $ \norm{f}_{C^2_+(E)} \leq M $. 
			%\item[] \textbf{QUERY:} $ x \in \R^2 $. 
			\item[] \textbf{RESULT}: A query function that accepts a point $ x \in \mathbb{R}^2 $ and produces a list of numbers $ ( f_\alpha(x) : \abs{\alpha} \leq 2) $ that guarantees the following: There exists a function $ F \in C^2_+(\R^2) $ with $ \norm{F}_{C^2(\R^2)} \leq CM$ and $ F|_E = f $, such that $ \d^\alpha F(x) = f_\alpha(x) $ for $ \abs{\alpha} \leq 2 $. The function $ F $ is independent of the query point $ x $.
			
			\item[] \textbf{COMPLEXITY:} 
			
			\begin{itemize}
				\item Preprocessing $ (E,f) $: at most $ CN\log N $ operations and $ CN $ storage.
				
				\item Answer query: at most $ C\log N $ operations.
			\end{itemize}
			
		\end{itemize}

	\end{algorithm}

	%The algorithm preprocesses the set $ E $ using at most $ CN\log N $ operations and $ CN $ storage. After receiving the query $ x \in \R^2 $, the algorithm processes the set $ S(x) $ associated with $ x $ using at most $ C\log N $ operations. After that, the algorithm computes the list $ ( f_\alpha(x) : \abs{\alpha} \leq 2) $ using at most $ C $ operations.

%	\begin{remark}
%		We will refer to $ C' $ in Theorem \ref{thm.bd}(B) as the \underline{depth} of the operator $ \E $. 
%	\end{remark}

	%The explanation for these algorithms will be given in \cite{JL20-Alg}.

	To conclude the introduction, we sketch the proof of Theorem \ref{thm.bd}, sacrificing accuracy for the ease of understanding. 
	
	As in \cite{JL20}, we perform a Calder\'on-Zygmund decomposition for $ \rt $ into a collection of squares $ \Lambda = \set{Q} $, such that $ E $ localized to each neighborhood of $ Q\in \Lambda $ lies on a curve with slope $ \leq C $ and curvature $ \leq C\delta_Q^{-1} $, where $ \delta_Q $ is the sidelength of $ Q $. Furthermore, we require that nearby squares in $ \Lambda $ are comparable in sizes, so that a partition of unity can be applied to patch together local extensions. 
	
	On each $ Q \in \Lambda $, we construct an extension operator in the following way: First we treat the local data as one-dimensional, for which there already exists a bounded extension operator of bounded depth\cite{JL20}. Next, we extend in the transversal direction by constant. To guarantee Whitney compatibility among nearby extensions, we require the local extension operator to take a prescribed jet near the center of $ Q $. The prescribed jet will be \qt{universal} in the sense that, up to a Taylor error, it is a candidate jet of nearby extensions. 
	
	Having accomplished all of the above, we patch together each of the local extension operators by a partition of unity. 
	
	Here we have given an overly simplified account of our approach. In practice, we have to simultaneously control derivatives on small scales and handle subtraction with great care, in order to preserve nonnegativity. We will present the details in Section \ref{S.bd}.

	The dependence on the parameter $ M $ in Theorem \ref{thm.bd} is a manifestation of the delicacy in preserving nonnegativity. We hereby raise the following question.

	\begin{problem}
		Let $m,n $ be positive integers. Let $ E \subset \Rn $ be a finite set. Do there exist constants $C(m,n)$, $D(m,n)$, and a map $ \E : C^m_+(E) \to C^m_+(\Rn) $ such that the following hold?
		\begin{enumerate}[(A)]
			\item $ \E(f) = f $ on $ E $ and $ \norm{\E(f)} \leq C\norm{f}_{C^m_+(E)} $ for all $ f \in C^m_+(E) $.
			\item For each $ x \in \Rn $, there exists $ S(x) \subset E $ with $ \#(S(x)) \leq D $ such that for all $ f, g \in C^m_+(E) $ with $ f|_{S(x)} = g|_{S(x)} $, we have
			\begin{equation*}
			\d^\alpha \E(f)(x)=\d^\alpha\E(g)(x)\for\abs{\alpha} \leq 2.
			\end{equation*}
		\end{enumerate}
	\end{problem}

	 This paper is part of a literature on extension and interpolation, going back to the seminal works of H.~Whitney\cite{W34-1,W34-2,W34-3}. We refer the readers to \cite{FK09-Data-1,FK09-Data-2,F09-Data-3,JL20,FIL16,FIL16+,FI20-book} and references therein for the history and related problems.

	\begin{acknowledgment}
		We are indebted to Charles Fefferman, Kevin O'Neill, and Pavel Shvartsman for their valuable comments. We also thank all the participants in the 11th Whitney workshop for fruitful discussions, and Trinity College Dublin for hosting the workshop. 
		
		The first author is supported by the UC Davis Summer Graduate Student Researcher Award and the Alice Leung Scholarship in Mathematics. The second author is supported by NSF Grant DMS-1554733 and the UC Davis Chancellor's Fellowship. 
	\end{acknowledgment}

	\section{Preliminaries}
	
	\newcommand{\touch}{\leftrightarrow}
	
	We use $ c_*, C_*, C' $, etc. to denote universal constants. They may be different quantities in different occurrences. We will label them to avoid confusion when necessary.
	
	We assume that we are given an ordered orthogonal coordinate system on $ \rt $. We use $ \abs{\,\cdot\,} $ to denote Euclidean distance. We use $ B(x,r) $ to denote the disk of radius $ r $ centered at $ x $. Let $ A, B \subset \rt $, we write $ \dist{A}{B} := \inf_{a \in A, b \in B}\abs{a - b} $. 
	
	By a square, we mean a set of the form $ Q = [a,a+\delta) \times [b,b+\delta) $ for some $ a,b \in \R $ and $ \delta > 0 $. If $ Q $ is a square, we write $ \dq $ to denote the sidelength of the square. For $ \lambda > 0 $, we use $ \lambda Q $ to denote the square whose center is that of $ Q $ and whose sidelength is $ \lambda\dq $. Given two squares $ Q, Q' $, we write $ Q\touch Q' $ if $ closure(Q) \cap closure(Q') \neq \void $.
	
	A dyadic square is a square of the form $ Q = [2^k \cdot i, 2^k\cdot (i + 1) ) \times [2^k\cdot j, 2^k\cdot j + 1) $ for some $ i,j,k \in \mathbb{Z} $. Each dyadic square $ Q $ is contained in a unique dyadic square with sidelength $ 2\dq $, denoted by $ Q^+ $.

	We use $ \alpha,\beta \in \mathbb{N}_0^2 $ to denote multi-indices. We adopt the partial ordering $ \alpha \leq \beta $ if and only if $ \alpha_i \leq \beta_i $ for $ i = 1,2 $. 
	
	Let $ \Omega \subset \Rn $ be a set with nonempty interior $ \Omega_0 $. For nonnegative integers $ m,n $, we use $ C^m(\Omega) $ to denote the vector space of $ m $-times continuously differentiable real-valued functions up to the closure of $ \Omega $. For $ F \in C^m(\Omega) $, we define
	\begin{equation*}
	\norm{F}_{C^m(\Omega)} := \sup_{x \in \Omega_0} \brac{\sum_{\abs{\alpha} \leq m}\abs{\d^\alpha F(x)}^2}^{1/2}.
	\end{equation*}
	We write $ C^m_+(\Omega) $ to denote the collection of functions $ F \in C^m(\Omega) $ such that $ F \geq 0 $ on $ \Omega $. 
	
	Let $ E \subset \Rn $ be finite. We define the following.
	\begin{equation*}
	\begin{split}
	C^m(E) := \set{F|_E : F \in C^m(\Rn)} \cong \R^{\#(E)}
	\enskip&\text{ and }\enskip
	\norm{f}_{C^m(E)}  := \inf\set{\norm{F}_{C^m(\Rn)} : F |_E = f};\\
	C^m_+(E) := \set{F|_E : F \in C^m_+(\Rn)} \cong \pos^{\#(E)}
	\enskip&\text{ and }\enskip
	\norm{f}_{C^m_+(E)} := \inf\set{\norm{F}_{C^m(\Rn)} : F|_E = f
		\text{ and }F\geq 0}\,.
	\end{split}
	\end{equation*}

	\subsection{Convex sets and polynomials}
	
	\newcommand{\ring}{\mathcal{R}}

	We write $ \P $ to denote the space of affine polynomials on $ \rt $. It is a three-dimensional vector space. For $ x \in \R^2 $ and a function $ F $ differentiable at $ x $, we write $ \jet_x F $ to denote the one-jet of $ F $ at $ x $, which we identify with the degree-one Taylor polynomial
	\begin{equation*}
	\jet_x F(y) := F(x) + \grad F(x) \cdot (y-x)\,.
	\end{equation*}
	We use $ \ring_x $ to denote the ring of one-jets at $ x $. For $ P \in \ring_x $, we define
	\begin{equation}
		\abs{P}_{\ring_x} := \brac{\sum_{\abs{\alpha} \leq 1}\abs{\d^\alpha P(x)}^2}^{1/2}.
		\label{eq.ring-norm}
	\end{equation}
	
	%We can identify $ \ring_x $ with $ \R^3 $ using the map $ \jet_x F \mapsto (F(x), \nabla F(x)) $. 
	
	\newcommand{\wt}{W^2}
	\newcommand{\wtp}{W^2_+}
	Let $ S \subset \Rn $ be a nonempty finite set. A \underline{Whitney field} on $ S $ is an array of polynomials
	\begin{equation*}
	\vec{P} := (P^x)_{x \in S}
	\enskip\text{ where }
	\enskip 
	P^x \in \ring_x
	\text{ for each } x \in S\,.
	\end{equation*}
	Given $ \vec{P} = (P^x)_{x \in S} $, we sometimes use the notation
	\begin{equation*}
	(\vec{P},x) := P^x
	\for x \in S\,.
	\end{equation*}
	We write $ \wt(S) $ to denote the vector space of all Whitney fields on $ S $. For $ \vec{P} = (P^x)_{s \in S} \in \wt(S) $, we define
	\begin{equation*}
	\norm{\vec{P}}_{\wt(S)} := \max_{\substack{x, y \in S,\,x \neq y,\,\abs{\alpha}\leq 2}}\frac{\abs{\d^\alpha(P^x - P^y)(x)}}{\abs{x - y}^{2-\abs{\alpha}}}.
	\end{equation*}
	$ \norm{\cdot}_{\wt(S)} $ is a seminorm on $ \wt(S) $. 
	
	We write $ \wtp(S) $ to denote a subcollection of $ \wt(S) $, such that $ \vec{P} \in W^2_+(S) $ if and only if for each $ x \in S $, there exists some $ M_x \geq 0 $ such that 
	\begin{equation*}
	(\vec{P},x)(y) + M_x\abs{y - x}^2 \geq 0
	\text{ for all } y \in \rt\,.
	\label{eq.MP}
	\end{equation*}
	For $ \vec{P} \in \wtp(S) $, we define
	\begin{equation*}
	\norm{\vec{P}}_{\wtp(S)} := \norm{\vec{P}}_{\wt(S)} + \max_{x \in S} \brac{\inf \set{ M_x \geq 0 : (\vec{P},x)(y) + M_x\abs{y-x}^2 \geq 0 \, \text{ for all } y \in \rt }}.
	\end{equation*}
	%Here we use the convention $ \frac{0}{0} = 0 $. 

	The next lemma is a Taylor-Whitney correspondence for $ \ctp(\rt) $. (A) is simply Taylor's theorem. See \cite{JL20,FIL16+} for a proof of (B). 
	
	\begin{lemma}\label{lem.WT}
		There exists a universal constant $ C_w $ such that the following holds.
		
		Let $ E \subset \rt $ be a finite set. 
		
		\begin{enumerate}[(A)]
			\item Let $ F \in \ctp(\rt) $. Let $ \vec{P} := (\jet_x F)_{x \in E} $. Then $ \vec{P} \in \wtp(E) $
			and
			$ \norm{\vec{P}}_{\wtp(E)} \leq C_w\norm{F}_{\ct(\rt)} $. 
			\item There exists a map $ T_w^E : \wtp(E) \to \ctp(\rt) $ such that $ \norm{T_w^E(\vec{P})}_{\ct(\rt)} \leq C_w\norm{\vec{P}}_{\wtp(E)} $ and $ \jet_x T_w^E(\vec{P}) = (\vec{P},x) $ for each $ x \in E $. 
			%The structure of $ T_w $ depends on the set $ E $, and we write $ T_w(E) $ instead of $ T_w $ to emphasize such dependence when necessary. 
		\end{enumerate}
	\end{lemma}

	%Lemma \ref{lem.WT}(A) is simply Taylor's theorem. See \cite{JL20} for the proof of (B). 

	\newcommand{\s}{\sigma}
	\newcommand{\sk}{\sigma^\sharp}
	\newcommand{\G}{\Gamma_+}
	\newcommand{\Gk}{\Gamma_+^\sharp}

	\begin{definition}
		Let $ E \subset \rt $ be a finite set. 
		
		\begin{itemize}
			\item For $ x \in \rt $, $ S \subset E $, and $ k \geq 0$, we define
			\begin{equation*}
			\begin{split}
			\s(x,S) &:= \set{\jet_x\phi : \phi \in \ct(\rt),\enskip
				\phi|_S = 0, \text{ and }
				\norm{\phi}_{\ct(\rt)} \leq 1}, \text{ and }\\
			\sk(x,k) &:= \bigcap_{S \subset E, \#(S) \leq k} \s(x,S).
			\end{split}
			\end{equation*}
			
			\item Let $ f : E \to \pos $ be given. For $ x \in \rt $, $ S \subset E $, $ k \geq 0$, and $ M \geq 0 $, we define
			\begin{equation*}
			\begin{split}
			\G(x,S,M) &:= \set{\jet_x F : F \in \ctp(\rt),\enskip
				F|_S = f, \text{ and }
				\norm{F}_{\ct(\rt)} \leq M}, \text{ and }\\
			\Gk(x,k,M) &:= \bigcap_{S \subset E, \#(S) \leq k}\G(x,S,M).
			\end{split}
			\end{equation*}
		\end{itemize}

	\end{definition}

	%\begin{remark}We will use $ \sk $ to detect the local geometry of $ E $, and measure the maximal amount of freedom to perturb a local extension. The actual amount of freedom, however, may be much smaller, due to the nonnegative constraint. We will need to compute the (approximate) shape of $ \sk $. \end{remark}

	%\begin{remark}We will use $ \Gk $ to control the derivatives of our extension. However, we do {\em not} need to compute the (approximate) shape of $ \Gk $. \end{remark}

	The next lemma characterizes jets generated by nonnegative functions.

	\begin{lemma}[Lemma 7.2 of \cite{JL20}]\label{lem.TC}
		Suppose $ P \in \G(x,\void,M) $. Then 
		\begin{align}
			P(y) + CM\abs{y - x}^2 &\geq 0\text{ for all } y \in \rt.\\
			\abs{\nabla P} &\leq C'\sqrt{MP(x)}.\label{eq.TC2}\\
			\dist{x}{\set{P = 0}} &\geq C''M^{-1/2}\sqrt{P(x)}.
		\end{align}
	\end{lemma}

	We prove a variant of Lemma \ref{lem.TC} for locally defined functions with \qt{small} values.
	
	\begin{lemma}\label{lem.TC-small}
		Let $ Q \subset \rt $ be a square. Let $ x_0 \in 2Q $. Let $ F \in \ctp(\rt) $ with $ \norm{F}_{\ct(\rt)} \leq M $. Suppose $ F(x_0) \leq M\dq^2 $. Then the following hold.
		\begin{enumerate}[(A)]
			\item $ \abs{\d^\alpha F(x)} \leq CM\dq^2 $ for $ \abs{\alpha} \leq 2 $ and $ x \in 100Q $. 
			\item $ \jet_{x_0} F \in \G(x_0,\void,CM) $.
		\end{enumerate}

	\end{lemma}

	\begin{proof}
		
		Without loss of generality, we may assume that $ x_0 = 0 $. 
		
		We prove (A) first.
		
		By Taylor's theorem, we see that, for $ x \in 100Q $,
		\begin{equation*}
		F(x) -\jet_0 F(x) \leq \abs{F(x) - \jet_0 F(x)} \leq C\abs{x}^2\cdot\sup_{y \in 100Q, \abs{\alpha} = 2}\abs{\d^\alpha F(y)} \leq CM\abs{x}^2.
		\end{equation*}
		Rearranging and expanding $ \jet_0 F(x) = F(0) + \nabla F(0)\cdot x $, we see that
		\begin{equation*}
		F(x) = F(0) + \nabla F(0)\cdot x +  CM\abs{x}^2 \geq 0
		\for x \in 100Q.
		\end{equation*}
		In particular, we see that
		\begin{equation}
		\nabla F(0) \cdot x \geq -F(0)- CM\dq^2 \geq C'M\dq^2
		\for x \in 100Q.
		\label{eq.2.1.4}
		\end{equation}
		Let $ x $ range over all the possible directions in $ \rt $ in \eqref{eq.2.1.4}. We see that
		\begin{equation}
		\abs{\nabla F(0)} \leq CM\dq.
		\label{eq.2.3.2}
		\end{equation}
		Thanks to \eqref{eq.2.3.2} and the assumption $ \abs{F(0)} \leq M\dq^2 $, we have
		\begin{equation}
		\abs{\d^\alpha F(0)} \leq CM\dq^{2-\abs{\alpha}}
		\for 
		\abs{\alpha} \leq 2.
		\label{eq.2.3.1}
		\end{equation}
		Applying the fundamental theorem of calculus to \eqref{eq.2.3.1}, we see that
		\begin{equation}
		\abs{\d^\alpha F(x)} \leq CM\dq^{2-\abs{\alpha}}
		\for
		\abs{\alpha} \leq 2, x \in 100Q.
		\label{eq.2.1.2}
		\end{equation}
		This proves Lemma \ref{lem.TC-small}(A). 
		
		Now we prove Lemma \ref{lem.TC-small}(B). 
		
		Let $ \psi $ be a cutoff function such that 
		\begin{equation}\label{eq.2.1.3}
		0 \leq \psi \leq 1,\,
		\psi \equiv 1 \text{ near }0,\,
		\supp{\psi} \subset 100Q, \text{ and }
		\abs{\d^\alpha \psi} \leq C\dq^{-\abs{\alpha}}.
		\end{equation}
		Consider the function
		\begin{equation*}
		\tilde{F}(x) := \psi(x) \cdot F(x).
		\end{equation*}
		Since $ \supp{\psi} \subset 100Q $, $ \tilde{F} $ is defined on all of $ \rt $ and $ \tilde{F} \geq 0 $. By \eqref{eq.2.1.2} and \eqref{eq.2.1.3}, we have 
		\begin{equation*}
		 \norm{\tilde{F}}_{\ct(\rt)}\leq CM.
		\end{equation*}
		Therefore, we have $ \jet_0F\in \G(0,\void,CM) $. This concludes part (B) and the proof of Lemma \ref{lem.TC-small}.

	\end{proof}

	\newcommand{\B}{\mathcal{B}}
	\begin{definition}
		For $ x \in \rt $ and $ \delta > 0 $, we define
		\begin{equation*}
		\B(x,\delta) := \set{P \in \P : \abs{\d^\alpha P(x)} \leq \delta^{2-\abs{\alpha}}\for \abs{\alpha}\leq 1.}. 
		\end{equation*}
	\end{definition}
	The significance of the object $ \B(x,\delta) $ is that Taylor's theorem can be reformulated as 
	\begin{equation*}
	\jet_x F - \jet_y F \in C\norm{F}_{\ct(\rt)} \B(x,\abs{x-y})
	\text{ for all } x,y \in \rt.
	\end{equation*}

	\namedtheorem[Helly's Theorem]{\label{thm.helly}Let $ \mathcal{F} $ be a finite collection of convex sets in $ \R^D $. Suppose every subcollection of $ \mathcal{F} $ of cardinality at most $ (D+1) $ has nonempty intersection. Then the whole collection has nonempty intersection.}{}
	
	The next lemma is a variant of Helly's Theorem. The proof can be found in \cite{F05-L}. 
	
	\begin{lemma}\label{lem.helly-v}
		Let $ \mathcal{F} $ be a finite collection of compact, convex, and symmetric subsets of $ \R^D $. Suppose $ 0 $ is an interior point for each $ K \in \mathcal{F} $. Then there exist $ K_1, \cdots, K_{D(D+1)} \in \mathcal{F} $ such that
		\begin{equation*}
		K_1 \cap \cdots \cap K_{D(D+1)} \subset C_D \cdot \brac{\bigcap\limits_{K \in \mathcal{F}} K } \,.
		\end{equation*}
		Here, $ C_D $ is a constant that depends only on $ D $.
	\end{lemma}

	\subsection{Calder\'on-Zygmund squares}
	\label{subsec.CZ}
	
	\newcommand{\Lz}{\Lambda_0}
	
	For the rest of the paper, we fixed the following constants
	\begin{equation*}
	k_0 \geq 4
	\enskip
	\text{ and }
	\enskip
	C_0 \geq 1000\,.
	\end{equation*}

	\begin{definition}\label{def.Lz}
		Let $ E \subset \rt $ be given. A dyadic square $ Q \in \Lz $ if and only if all of the following are satisfied.
		\begin{itemize}
			\item $ \dq \leq 1 $.
			\item $ \diam_{\ring_x}{\sk(x,k_0)} \geq C_0\delta_Q $ for all $ x \in 2Q $. Here and below, the metric on $ \ring_x $ is induced by \eqref{eq.ring-norm}. 
			
%			\footnote{Here and below, we identify $ \ring_{x} \cong \R^3 $ for each $ x_0 \in \rt $ via the map\begin{equation*}
%				\jet_{x}F \mapsto (F(x), \d_1F(x), \d_2F(x)).
%				\end{equation*}
%			For each $ x \in \rt $, we measure the diameter of a set in $ \ring_x $ in terms of its image in $ \R^3 $. 
%		}
			
			\item There exists $ \hat{x} \in 2Q^+ $ such that $ \diam_{\ring_{\hat{x}}}{\sk(\hat{x},k_0)} < 2C_0\delta_Q $. 
		\end{itemize}
	\end{definition}

	We will be mostly interested in the \qt{nontrivial} squares.
	
	\begin{definition}\label{def.Ls}
		We define
		\begin{equation*}
		\Ls := \set{Q \in \Lz : E \cap 2Q \neq \void}.
		\end{equation*}
	\end{definition}

	We include all the relevant properties of the squares in $ \Lz $ and $ \Ls $ in the following lemma. The proof for these properties can be found in \cite{JL20}.

	\newcommand{\xqs}{{x_Q^\sharp}}
	
	\begin{lemma}\label{lem.CZ}
		Let $ \Lz $ and $ \Ls $ be as in Definitions \ref{def.Lz} and \ref{def.Ls}, respectively. Then we have the following.
		
		\begin{enumerate}[(A)]
			\item\label{lem.CZ.cover} {\rm(Lemma 5.1 of \cite{JL20})} If $ Q, Q' \in \Lz $, and $ Q\touch Q'$, then $ \frac{1}{4}\delta_Q \leq \delta_{Q'} \leq 4\delta_Q $. As a consequence, we have the following bounded intersection property: For each $ Q \in \Lz $,
			\begin{equation}
			\eqindent
			\#\set{Q' \in \Lz : \frac{9}{8} Q' \cap \frac{9}{8}Q \neq \void} \leq C.
			\label{eq.bip}
			\end{equation}
			Moreover, we have
			\begin{equation*}
			\#(\Ls) \leq C\cdot \#(E).
			\end{equation*}

			\item\label{lem.CZ.graph} {\rm(Lemmas 5.4 and 5.5 in \cite{JL20})} Let $ Q \in \Ls $. Then up to a rotation, there exists $ \phi \in C^2(\R)$ such that 
			\begin{equation*}
			E \cap 2Q \in \set{(t,\phi(t)) : t\in \R}.
			\end{equation*}
			The orientation of the graph is given by the direction in which $ closure(\sk(x,k_0)) $ achieves $ \diam_{\ring_x}{\sk(x,k_0)} $ for some $ x \in E \cap 2Q $. Moreover, the function $ \phi $ satisfies the estimates
			\begin{equation*}
			\abs{\ddtm\phi(t)} \leq \delta_Q^{1-m}
			\for
			m = 1,2.
			\end{equation*}
			As a consequence, there exists a diffeomorphism $ \Phi : \rt\to\rt $ defined by 
			\begin{equation*}
			\Phi(t_1,t_2) \sim (t_1,t_2 -\phi(t_1)).
			\end{equation*}
			The symbol $ \sim $ means the definition is up to a rotation about the origin. In particular, $ \Phi $ satisfies
			\begin{equation*}
			\Phi(E \cap 2Q) \subset \R\times \set{t_2 = 0}
			\text{ and }
			\abs{\nabla^m\Phi}, |\nabla^m\Phi^{-1}| \leq C\delta_Q^{1-m}
			\for m = 1,2.
			\end{equation*}

			\item\label{lem.CZ.rep} {\rm(Lemma 5.6 in \cite{JL20})} Let $ Q \in \Lz $. There exists $ \xqs \in Q $ such that
			\begin{equation*}
			\dist{\xqs}{E} \geq c_0\delta_Q.
			\end{equation*}
			Here, $ c_0 $ is a universal constant.

		\end{enumerate}

	\end{lemma}

	\begin{lemma}\label{lem.comp}
		Let $0  < \delta \leq 1 $. Let $ \Phi : \rt\to \rt $ be a $ C^2 $-diffeomorphism such that
		\begin{equation*}
		\abs{\nabla^m\Phi} \leq \delta^{1-m}
		\for m = 1,2.
		\end{equation*}
		Let $ \Omega \subset \rt $ be a domain, and let $ F \in \ct(\Omega) $. Suppose
		\begin{equation*}
		\abs{\d^\alpha F(x)} \leq M\dq^{2-\abs{\alpha}}\for
		\abs{\alpha} \leq 2, x \in \Omega.
		\end{equation*}
		Then $ \norm{F\circ \Phi}_{\ct(\Phi^{-1}(\Omega))} \leq CM $.
	\end{lemma}

	\begin{proof}
		We expand $ \Phi = (\Phi_1, \Phi_2) $ in coordinates. Then\begin{equation*}
		\d_{ij}F\circ \Phi = \sum_{k,l = 1}^2 c_{kl}\cdot \d_i\Phi_k\cdot \d_j\Phi_l \cdot\d_{kl}F\circ \Phi + \sum_{k = 1}^2\d_{ij}\Phi_k \cdot \d_kF\circ \Phi.
		\end{equation*}
		Lemma \ref{lem.comp} follows from the derivative estimates of $ F $ and $ \Phi $.
	\end{proof}

	\begin{lemma}[Lemma 5.7 in \cite{JL20}]\label{lem.sk-ball}
		Let $ \Ls $ be as in Definition \ref{def.Ls}. Let $ Q \in \Ls $, and let $ \xqs $ be as in Lemma \ref{lem.CZ}(\ref{lem.CZ.rep}). Then 
		\begin{equation*}
		\eqindent
		\sk(\xqs,4k_0) \subset C\cdot \B(\xqs,\delta_Q),
		\end{equation*}
		for some universal constant $ C $.
	\end{lemma}

	\section{Extension operator of bounded depth}
	\label{S.bd}

	We use $ t $ to denote the variable on $ \R $. We write $ \d^m $ to denote $ \ddtm $. 
	
	We recall the following theorems proven in \cite{JL20}.
	
	\newcommand{\Eb}{\overline{\E}}
	\newcommand{\Ebpm}{\overline{\E}_{\pm}}
	
	\begin{subtheorem}{theorem}
		
		\renewcommand{\ddtm}{\d^m}
		
		\begin{theorem}[Theorem 2.A in \cite{JL20}]\label{thm.bd-1d}
			Let $ E \subset \R $ be a finite set. There exist universal constants $ C,D $ and an operator $ \Eb : C^2_+(E) \to C^2_+(\R) $ such that the following hold.
			\begin{enumerate}[(i)]
				\item $ \Eb (f) \big|_E = f $ for all $ f \in C^2_+(E) $.
				\item $ \norm{\Eb (f)}_{C^2(\R)} \leq C\norm{f}_{C^2_+(E)} $.
				\item 
				
				Moreover, for each $ t \in \R $, there exists $ S(t) \subset E $ with $ \#(S(t)) \leq D $, such that for all $ f,g \in C^2_+(E) $ with $ f|_{S(t)} = g|_{S(t)} $, we have
				\begin{equation*}
				\eqindent
				\ddtm (\Eb(f))(t) = \ddtm(\Eb (g))(t)
				\enskip
				\text{ for }
				m = 0,1,2\,.
				\end{equation*}

			\end{enumerate}
			
		\end{theorem}

		\begin{theorem}[Theorem 2.B. in \cite{JL20}]\label{thm.bd-1dpm}
			Let $ E \subset \R $ be a finite set. There exist universal constants $ C,D $ and a linear operator $ \Ebpm : C^2(E) \to C^2(\R) $ such that the following hold.
			\begin{enumerate}[(i)]
				\item $ \Ebpm (f) \big|_E = f $ for all $ f \in C^2(E) $.
				\item $ \norm{\Ebpm(f)}_{{C}^2(\R)} \leq C\norm{f}_{{C}^2(E)} $.

				\item Moreover, for each $ t \in \R $, there exists $ S(t) \subset E $ with $ \#(S(t)) \leq D $, such that for all $ f,g \in \ct(E) $ with $ f|_{S(t)} = g|_{S(t)} $, we have
				\begin{equation*}
				\ddtm (\Ebpm(f))(t) = \ddtm(\Ebpm (g))(t)
				\enskip
				\text{ for }
				m = 0,1,2\,.
				\end{equation*}

			\end{enumerate}
			
		\end{theorem}

	\end{subtheorem}

	\begin{remark}\label{rem.St}
		The set $ S(t) $ in Theorems \ref{thm.bd-1d}(iii) and \ref{thm.bd-1dpm}(iii) takes a particular simple form:
		\begin{itemize}
			\item Suppose $ \#(E) \leq 3 $. We take $ S(t) = E $.
			\item Suppose $ \#(E) \geq 4 $. Enumerate $ E = \set{t_1, \cdots, t_N} $ with $ t_1 < \cdots < t_N $.
			\begin{itemize}
				\item If $ t < t_1 $ or $ t> t_N $, we take $ S(t) $ to be the three points in $ E $ closest to $ x $.
				\item If $ t \in [t_1, t_2] $, we take $ S(t) = \set{t_1, t_2, t_3} $. 
				\item If $ t \in [t_{N-1}, t_N] $, we take $ S(t) = \set{t_{N-2}, t_{N-1}, t_N} $.
				\item Otherwise, we take $ S(t) = \set{t_1', t_2', t_3', t_4'} \subset E $ with $ t_1' < t_2' < t_3' < t_4' $ such that $ t \in [t_2', t_3'] $.
			\end{itemize}
		\end{itemize}
		
	\end{remark}
	
%	\begin{remark}
%		We will refer to $ C' $ in Theorem \ref{thm.bd-1d}(iii) and Theorem \ref{thm.bd-1dpm}(iii) as the \underline{depth} of the operators $ \Eb $ and $ \Ebpm $. 
%	\end{remark}
	
	As a consequence of \hyperref[thm.helly]{Helly's Theorem} and Lemma \ref{lem.helly-v}, we have the following.

	\begin{lemma}[Lemma 4.5 of \cite{JL20}]\label{lem.sk-loc}
		Let $ E \subset \rt $ be finite. Let $ x \in  \rt $. Let $ k \geq 0 $. Then there exist $ S_\nu = S_\nu(x,k) \subset E $ for $ \nu = 1, \cdots, 12 $ with $ \#(S_\nu) \leq k $ for each $ \nu $, and
		\begin{equation*}
		\bigcap_{\nu = 1}^{12} \s(x,S_\nu) \subset C \cdot \sk(x,k)\,.
		\end{equation*}
	\end{lemma}

	\begin{lemma}[Lemma 5.3 of \cite{JL20}]\label{lem.G-G}
		
		Let $ E \subset \rt $ be a finite set. Let $ \Lz $ be as in Definition \ref{def.Lz}. Let $ Q, Q' \in \Lz $. Let $ x_Q \in Q $ and $ x_{Q'} \in Q' $. Let $ P_Q \in \Gk(x_Q, 4k_0,M) $ and $ P_{Q'} \in \Gk(x_{Q'}, 4k_0,M) $. Then
		\begin{equation*}
		\abs{\d^\alpha (P_Q - P_{Q'})(x)} \leq CM \brac{\abs{x_Q - x_{Q'}} + \dq + \delta_{Q'}}^{2-\abs{\alpha}}
		\for x \in 2Q\cup 2 Q'
		\text{ and }\abs{\alpha} \leq 2\,.
		\end{equation*}
		 
	\end{lemma}

	\begin{lemma}[Lemma 7.3 of \cite{JL20}]\label{lem.perturb}
		Let $ E \subset \rt $ be finite. Let $ Q \in \Ls $ and let $ \xqs \in \Ls $. Let $ f : E \to \pos $ be given. Suppose $ \G(\xqs,4k_0,M) \neq \void $. The following are true. 
		\begin{enumerate}[(A)]
			\item There exists a number $ B_0 > 0 $ exceeding a large universal constant such that the following holds. Suppose $ f(x) \geq B_0M\dq^2 $ for each $ x \in E \cap 2Q $. Then 
			\begin{equation*}
			\eqindent
			\G(\xqs,4k_0,M) + M\cdot \B(\xqs,\dq)\subset \G(\xqs,4k_0,CM),
			\end{equation*}
			for some universal constant $ C $.
			\item Let $ A > 0 $. Suppose $ f(x) \leq AM\dq^2 $ for each $ x \in E \cap 2Q $. Then
			\begin{equation*}
			0 \in \G(\xqs,4k_0,A'M).
			\end{equation*}
			Here, $ A' $ depends only on $ A $. 
		\end{enumerate}
	\end{lemma}

	\newcommand{\ub}{\underline{B}}
	\newcommand{\ob}{\overline{B}}
	
	\begin{lemma}[Lemma 7.5 of \cite{JL20}]\label{lem.bs}
		For each $ \ub > 0 $, we can find $ \ob > 0 $ depending only on $ \ub $, such that the following holds.
		
		Let $ E \subset \rt $ be finite. Let $ Q \in \Ls $. Let $ f :E \to \pos $ be given. Suppose $ \G(\xqs,4k_0,M) \neq \void $. Then at least one of the following holds.
		\begin{enumerate}[(A)]
			\item $ f(x) \geq \ub M\delta_Q^2 $ for all $ x \in E \cap 2Q $.
			\item $ f(x) \leq \ob M\delta_Q^2 $ for all $ x \in E \cap 2Q $. 
		\end{enumerate}
		
	\end{lemma}

%	\begin{lemma}[Lemma 7.1 of \cite{JL20}]\label{lem.lip}
%		Let $ E \subset \rt $ be finite. Let $ \Lz $ be as in Section \ref{subsec.CZ}. Let $ Q \in \Lz $. Let $ f : E \to \pos $ be given. Suppose $ \Gk(\xqs,4k_0,M) \neq \void $. Then there exists $ F_Q \in \ctp(100Q) $ such that $ F|_{E \cap 2Q} = f $, $ \norm{F_Q}_{\ct(\rt)} \leq CM $, and $ \jet_\xqs F_Q \in \Gk(\xqs,4k_0,M) $. 
%		
%	\end{lemma}

	The finiteness principle in \cite{JL20} can be translated into the following.

	\begin{lemma}\label{lem.FP-G}
		Let $ E \subset \rt $ be finite. Let $ f : E \to \pos $. Suppose $ \Gk(x,4k_0,M) \neq \void $ for all $ x \in \rt $ (recall that we assume $ k_0 \geq 4 $). Then there exists $ F \in \ctp(\rt) $ with $ F|_E = f $ and $ \norm{F}_{\ct(\rt)} \leq CM $. 
	\end{lemma}

	\renewcommand{\Q}{\mathcal{Q}}
	\newcommand{\M}{\mathcal{M}}

	Let $ E \subset \rt $ be finite. Let $ x_0 \in \rt $. 
	
	Let $ S_\nu = S_\nu(x_0,4k_0) $, $ \nu = 1, \cdots, 12 $, be as in Lemma \ref{lem.sk-loc} (with $ x = x_0 $ and $ k = 4k_0 $). We define
	\begin{equation}
	S^{x_0} := S_1 \cup \cdots \cap S_{12} \cup \set{x_0}.
	\label{eq.Sx0}
	\end{equation}
	We immediately see, from an elementary calculation and Lemma \ref{lem.sk-loc}, that
	\begin{equation*}
	\#(S^{x_0}) \leq 48k_0 + 1
	\enskip
	\text{ and }
	\enskip
	\s(x_0,S^{x_0}) \subset C\cdot \sk(x_0,4k_0).
	\end{equation*}
	We define the following two functions:
	\begin{equation}
	\begin{split}
	\Q^{x_0}:  \underbrace{\P \times \cdots\times \P}_{\#(S^{x_0})\text{ copies}} &\to \pos\\
	\vec{P} = (P^x)_{x \in S^{x_0}} &\mapsto 
	\sum_{\substack{x \in S^{x_0}\\\abs{\alpha}\leq 1}}
	\abs{\d^\alpha P^x(x)}^{2} 
	+
	\sum_{\substack{x,y \in S^{x_0}\\x\neq y\\\abs{\alpha}\leq 1}}
	\brac{ \frac{\abs{\d^\alpha(P^x - P^y)(x)}}{\abs{x-y}^{2-\abs{\alpha}}}  }^2\,,\text{ and }
	\end{split}
	\label{eq.Qx}
	\end{equation}
	\begin{equation}
	\begin{split}
	\M^{x_0}: \underbrace{\P \times \cdots\times \P}_{\#(S^{x_0})\text{ copies}} &\to [0,\infty]\\
	\vec{P} = (P^x)_{x \in S^{x_0}} &\mapsto
	\begin{cases}
	\infty &\text{ if there exists } x \in S^{x_0}  \text{ such that } P^{x}(x) < 0\\
	\sum\limits_{x \in S^{x_0}} \frac{\abs{\grad P^x}^4}{P^x(x)^2}
	&\text{ otherwise}
	\end{cases}\,.
	\end{split}
	\label{eq.Mx}
	\end{equation}
	In \eqref{eq.Mx}, we use the convention $ \frac{0}{0}= 0 $.
	
%	We note that both $ \Q^{x_0} $ and $ \M^{x_0} $ have positive-semidefinite Hessians 
	
	It is clear that for each $ \vec{P} \in \wtp(S^{x_0}) $, we have
	\begin{equation*}
	C^{-1}(\Q^{x_0} + \M^{x_0})(\vec{P}) \leq \norm{\vec{P}}_{\wtp(S^{x_0})}^2 \leq C(\Q^{x_0} + \M^{x_0})(\vec{P}).
	\end{equation*} 
	
	\newcommand{\Gg}{\mathcal{G}}
	Let $ f : E \to \pos $ be given. For each $ M \geq 0 $, the functions $ \Q^{x_0} $ and $ \M^{x_0} $ give rise to a (possibly empty) set
	\begin{equation}
	\Gg(x_0,M) := \set{P \in \P : \begin{matrix}&\text{There exists }\vec{P} \in \wtp(S^{x_0})\text{ such that }\\
		&(\vec{P},x_0) = P,\\
		&\Q^{x_0}(\vec{P}) + \M^{x_0}(\vec{P}) \leq C_TM^2
		,
		\text{ and }\\
		&(\vec{P},x)(x) = f(x) 
		\text{ for each } x \in S^{x_0} \cap E\,.
		\end{matrix}
	}.
	\label{eq.Gg}
	\end{equation}
	Here, $ C_T $ is a large universal constant.
	
	%Since both $ \Q^{x_0} $ and $ \M^{x_0} $ have positive semidefinite Hessians away from the singularities of $ \M^{x_0} $, the set $ \Gg(x_0,M) $ in \eqref{eq.Gg} is convex. 
	
	We learn from Lemma \ref{lem.WT} that if $ \norm{f}_{\ctp(S^{x_0}\cap E)} \leq M $, then $ \Gg(x_0,M) \neq \void $. 
	
	Suppose $ \Gg(x_0,M) \neq \void $. Let $ {P}_0^{x_0} \in closure(\Gg(x_0,M)) $ be such that
	\begin{equation}
	\sum_{\abs{\alpha} \leq 1}\abs{\d^\alpha P_0^{x_0}(x_0)}^2 = \inf \set{ \sum_{\abs{\alpha} \leq 1}\abs{\d^\alpha P(x_0)}^2 : P \in \Gg(x_0,M)  }\,.
	\label{eq.P0}
	\end{equation}

	\begin{definition}\label{def.Tx0}
		Let $ E \subset \rt $ be finite, and let $ f : E \to \pos $. Let $ x_0 \in \rt $. Let $ \Gg(x_0,M) $ and $ P_0^{x_0} $ be as in \eqref{eq.Gg} and \eqref{eq.P0}. We define
		\begin{equation}
		\begin{split}
		T^{x_0} : \ctp(E) \times \pos &\to \P\\
		(f,M) &\mapsto \begin{cases}
		P_0^{x_0}&\text{ if } \Gg(x_0,M) \neq \void\\
		0 &\text{ otherwise}
		\end{cases}\,.
		\end{split}
		\label{eq.tx0}
		\end{equation}
	\end{definition}

	\begin{remark}
		For any given $ x_0 \in \rt $, $ T^{x_0}(M,f) $ depends only on $ f|_{S^{x_0} \cap E} $.

	\end{remark}

	The following lemma shows that for suitable $ M $, $ T^{x_0}(f,M) $ is the jet of a nonnegative interpolant of norm no greater than $ CM $ and agrees with $ f $ on $ S^{x_0} \cap E $.

	\begin{lemma}\label{lem.Fx0}
		Let $ E \subset \rt $ be finite, and let $ f : E \to \pos $. Let $ x_0 \in \rt $. Let $ S^{x_0} $ be as in \eqref{eq.Sx0}. Let $ T^{x_0} $ be as in Definition \ref{def.Tx0}. Suppose $ \norm{f}_{\ctp(E)} \leq M $. Then there exists $ F \in \ctp(\rt) $ with $ \norm{F}_{\ct(\rt)} \leq CM $, $ F(x) = f(x) $ for each $ x \in S^{x_0} \cap E $, and $ \jet_{x_0}F = T^{x_0}(f,M) $. 
	\end{lemma}

	\begin{proof}
		\newcommand{\sx}{S^{x_0}}
		By the definition of the trace norm, we have\begin{equation*}
		\norm{f}_{\ctp(\sx\cap E)} \leq
		\norm{f}_{\ctp(E)} \leq M\,. 
		\end{equation*}
		Thanks to Lemma \ref{lem.WT}, there exists
		$ \vec{P} = (P^x)_{x \in \sx} \in \wtp(\sx)  $
		with $ \norm{\vec{P}}_{\wtp(\sx)} \leq C_WM $ and $ P^x(x) = f(x) $ for each $ x \in \sx\cap E $. Then, for $ C_T \geq C_W $ (see \ref{eq.Gg}),
		\begin{equation*}
		\Gg(x_0,M) \neq \void\,.
		\end{equation*}
		By Definition \ref{def.Tx0}, 
		\begin{equation*}
		T^{x_0}(f,M)  \in closure(\Gg(x_0,M)) \subset \Gg(x_0,2M)\,.
		\end{equation*}
		By the definition of $ \Gg(x_0,M) $, there exists $ \vec{P}_0 \in \wtp(\sx) $ with 
		\begin{equation*}
		(\vec{P}_0,x_0) = T^{x_0}(f,M),\,
		\norm{\vec{P}}_{\wtp(\sx)} \leq CM,\,\text{ and }
		(\vec{P}_0,x)(x) = f(x)
		\text{ for each }
		x \in \sx\cap E\,.
		\end{equation*}
		Let $ T_w^{\{x_0\}} $ be as in Lemma \ref{lem.WT}. We define
		\begin{equation*}
		F := T_w^{\{x_0\}} (\vec{P}_0)\,.
		\end{equation*}
		Thus, $ F \in \ctp(\rt) $ with $ \norm{F}_{\ct(\rt)} \leq CM $ and $ \jet_x F = (\vec{P}_0,x) $ for each $ x \in \sx $. In particular, $ F(x) = f(x) $ for each $ x \in \sx\cap E $. This proves Lemma \ref{lem.Fx0}.
	\end{proof}

	The next lemma shows that for each representative point $ \xqs \in Q \in \Ls $, $ T^{\xqs}(f,M) $ is \qt{universal}.
	
	\begin{lemma}\label{lem.Txq}
		Let $ E \subset \rt $ be finite. Let $ Q \in \Ls $. Let $ \xqs \in Q $ be as in Lemma \ref{lem.CZ}(\ref{lem.CZ.rep}). Let $ f : E \to \pos $ be given. Suppose $ \norm{f}_{\ctp(E) } \leq M $. Then
		\begin{equation*}
		T^{\xqs}(f,M) \in \Gk(\xqs,4k_0,CM).
		\end{equation*}
	\end{lemma}
	
	\begin{proof}
		Since $ \norm{f}_{\ctp(E)} \leq M $, we have $ \G(\xqs,4k_0,C_0M) \neq \void $.

		Let $ \ub $ be a sufficiently large number to be determined. By Lemma \ref{lem.bs}, there exists $ \ob $ such that at least one of the following scenarios is true.
		\begin{enumerate}[\text{Scenario} 1]
			\item $ f(x) \geq \ub M\dq^2 $ for all $ x \in E \cap 2Q $.
			\item $ f(x) \leq \ob M\dq^2 $ for all $ x \in E \cap 2Q $. 
		\end{enumerate}
		
		%For the rest of the proof, we use $ C_1, C_2 $ etc. to denote constants that are allowed to depend on $ \ub $ and $ \ob $. 
		
		Suppose we are in the first scenario.

		\newcommand{\hf}{\hat{F}}
		Since $ \norm{f}_{\ctp(E)} \leq M $, there exists
		\begin{equation*}
		\hf \in \ctp(\rt)\text{ with } \norm{\hf}_{\ct(\rt)} \leq C_1M,\,
		\hf_E = f,\,\text{ and }
		\jet_{\xqs}\hf \in \G(\xqs,4k_0,C_1M)\,.
		\end{equation*}
		
		\newcommand{\tf}{\tilde{F}}
		\newcommand{\txq}{T^{\xqs}}
		Let $ \tf $ be as in Lemma \ref{lem.Fx0} with $ x_0 = \xqs $. Thus,
		\begin{equation*}
		\norm{\tf}_{\ct(\rt)} \leq C_2M,\,
		\tf\big|_{S^{\xqs}\cap E} = f,\,
		\text{ and }
		\jet_{\xqs}\tf = \txq(f,M).
		\end{equation*}
		
		Therefore, we have 
		\begin{enumerate}[(a)]
			\item $ \norm{\tf -\hf}_{\ct(\rt)} \leq C_3M $;
			\item $ (\tf - \hf)\big|_{S^{\xqs}\cap E} = 0 $; and
			\item $ \jet_{\xqs}(\tf - \hf) = \txq(f,M) - \jet_{\xqs}\hf $. 
		\end{enumerate}
		By (a), (b), the definition of $ \s $, and Lemma \ref{lem.sk-loc}, we have
		\begin{equation*}
		\jet_{\xqs}(\tf - \hf) \in C_3M\cdot \s(\xqs,S^{\xqs}\cap E) \subset C_4M\cdot \sk(\xqs,4k_0).
		\end{equation*}
		By Lemma \ref{lem.sk-ball}, we then have
		\begin{equation*}
		\jet_{\xqs}(\tf - \hf) \in C_5M\cdot \B(\xqs,\dq)\,.
		\end{equation*}
		Now, we pick 
		\begin{equation}
		\ub \geq B_0 \cdot \max\set{1,C_0,C_5}\,.
		\label{eq.ub}
		\end{equation}
		Here, $ B_0 $ is as in Lemma \ref{lem.perturb}(A). We then have
		\begin{equation*}
		\begin{split}
		\txq(f,M) &\in \jet_{\xqs}\hf + C_5M\cdot \B(\xqs,\dq)\\
		&\subset \Gk(\xqs,4k_0,C_1M) + C_5M\cdot\B(\xqs,\dq)\\
		&\subset \Gk(\xqs,4k_0,C_6M) + C_6M\cdot\B(\xqs,\dq)\\
		&\subset \Gk(\xqs,4k_0,C_7M).
		\end{split}
		\end{equation*}
		Note that for the last inclusion, we used the assumption in Scenario 1, the choice of $ \ub $ in \eqref{eq.ub}, and Lemma \ref{lem.perturb}.
		
		This proves Lemma \ref{lem.Txq} for the first scenario.
		
		Now we turn to the second scenario.

		By Lemma \ref{lem.perturb}, we know that
		\begin{equation}
		0 \in \Gk(\xqs,4k_0,BM)\,.
		\label{eq.3.7.1}
		\end{equation}
		Here, $ B $ depends only on $ \ob $. Thanks to Lemma \ref{lem.FP-G}, we know that
		\begin{equation}
		0 \in \Gk(\xqs,4k_0,BM) \subset 
		\G(\xqs,E,CBM)
		\subset \G(\xqs,S^{\xqs}\cap E,C'BM).
		\label{eq.3.7.2}
		\end{equation}
		
		Since $ \norm{f}_{\ctp(E)} \leq M $, we have $ \Gg(\xqs,M) \neq \void $, where $ \Gg(\xqs,M) $ is as in \eqref{eq.Gg}. 
		Taking $ C_T $ to be sufficiently large in \eqref{eq.Gg} and using Lemma \ref{lem.WT}, we see that
		\begin{equation}
		\G(\xqs,S^{\xqs}\cap E,C_1BM) \subset \Gg(\xqs,M).
		\label{eq.3.7.3}
		\end{equation}
		From \eqref{eq.3.7.1}, \eqref{eq.3.7.2}, and \eqref{eq.3.7.3}, we see that
		\begin{equation}
		0 \in  \Gg(\xqs,M) \subset closure(\Gg(\xqs,M)).
		\label{eq.3.7.4}
		\end{equation}
		Since $ \txq(f,M) $ is defined to be the element of the least norm in $ closure(\Gg(\xqs,M)) $, \eqref{eq.3.7.1} and \eqref{eq.3.7.4} imply
		\begin{equation*}
		\txq(f,M) = 0 \in \Gk(\xqs,4k_0,BM).
		\end{equation*}
		This concludes the second scenario as well as the proof of Lemma \ref{lem.Txq}.
		
	\end{proof}

	The next results shows how to construct a local extension operator of bounded depth that also takes a prescribed jet at $ \xqs \in Q \in \Ls $.

	\begin{lemma}\label{lem.ext-loc}
		Let $ E \subset \rt $ be finite. Let $ Q \in\Ls $. Then there exist universal constants $ C, D_0 $ and a map \begin{equation*}
		\E_Q : \ctp(E) \times [0,M) \to \ct(100Q)
		\end{equation*}
		such that the following hold.
		\begin{enumerate}[(A)]
			\item Given $ f \in \ctp(E) $ with $ \norm{f}_{\ctp(E)} \leq M $, we have \begin{enumerate}[{\rm(i)}]
				\item $ \E_Q(f,M) \geq 0$ on $ 100Q $;
				\item $ \E_Q(f,M) = f $ on $ E \cap 2Q $; and
				\item $ \norm{\E_Q(f,M)}_{\ct(100Q)} \leq CM $.
				\item $ \jet_{\xqs} \E_Q(f,M) = T^{\xqs}(f,M)$, where $ T^{\xqs} $ is as in Definition \ref{def.Tx0} (with $ x_0 = \xqs $, and $ \xqs $ is as in Lemma \ref{lem.CZ}(\ref{lem.CZ.rep})). As a consequence of Lemma \ref{lem.Txq}, we have $ \jet_{\xqs}\E_Q(f,M) \in \Gk(\xqs,4k_0,CM) $. 
				 .
			\end{enumerate}
			
			\item For each $ x \in 100Q $, there exists a set $ S_Q(x) \subset E $ with $ \#(S_Q(x)) \leq D_0 $ such that the following holds: Given $ f,g \in \ctp(E) $ with $ \norm{f}_{\ctp(E)},\norm{g}_{\ctp(E)} \leq M $ and $ f|_{S_Q(x)} = g|_{S_Q(x)} $, we have
			\begin{equation*}
			\d^\alpha \E_Q(f,M)(x) = \d^\alpha \E_Q(g,M)(x)
			\for\abs{\alpha}\leq 2\,.
			\end{equation*}
		\end{enumerate}
	\end{lemma}

	Before dwelling into the proof, which is rather lengthy, we briefly describe the strategy. 
	
	For part (A), we use Lemma \ref{lem.bs} to divide the situation at hand into two cases: one where the local data is \qt{sufficiently big} and one where the local data is \qt{sufficiently small}. For the big case, we can, without losing nonnegativity, subtract the prescribed jet from the local data, solve the straightened one-dimensional interpolation problem, and add back the prescribed jet. For the small case, we can directly solve the straightened one-dimensional problem and prescribe a zero jet. The reason behind the two separate methods is that we need to handle the following two matters simultaneously: preserve nonnegativity of the extension, and avoid large second derivatives from composition with the diffeomorphism. 
	
	We now present the details.

	\begin{proof}[Proof of Lemma \ref{lem.ext-loc}]
		
		First, we list the relevant ingredients in the construction of $ \E_Q $. 
		
		\begin{itemize}
			\item Let $ \Phi $ be the diffeomorphism associated with $ Q $ as in Lemma \ref{lem.CZ}(\ref{lem.CZ.graph}).
			
			\item Let $ \xqs,c_0 $ be as in Lemma \ref{lem.CZ}(\ref{lem.CZ.rep}). Let $ \psi \in \ctp(\rt) $ be a cutoff function such that
			\begin{equation}
			\eqindent
			0 \leq \psi \leq 1,\,
			\psi \equiv 1
			\text{ near }\xqs,\,
			\supp{\psi} \subset B(\xqs,c_0\delta_Q),\text{ and }
			\abs{\d^\alpha\psi}\leq C\dq^{2-\abs{\alpha}}.
			\label{eq.psi}
			\end{equation}
			
			\item Let $ T^{\xqs} $ be as in Definition \ref{def.Tx0}, with $ x_0 = \xqs $.  
			%We write $ P_Q^\sharp := T^\xqs(f,M) $.
			
			\item Define
			\begin{equation}
			\eqindent
			\Delta(f,M,Q) := \begin{cases}
			1 &\text{ if } {T^{\xqs}(f,M)} \text{ is not the zero polynomial}\\
			0 &\text{ otherwise}
			\end{cases}\,.
			\label{eq.dfmq}
			\end{equation}
			
			\item Let $ \Eb $ and $ \Ebpm $ be as in Theorems \ref{thm.bd-1d} and \ref{thm.bd-1dpm}. 
			
			\item Let $ V $ be the map defined by $ V(g)(s,t) := g(s) $ for all $ g : \R \supset I \to \R $.

		\end{itemize}

		We prove (A) first.

		We define
		\begin{equation}
		\begin{split}
		\E_Q(f,M) &:= {T^{\xqs}(f,M)} + (1-\psi)\cdot \widetilde{\mathcal{E}}_Q(f,M),\,\text{ where }\\
		\widetilde{\mathcal{E}}_Q(f,M)
		&:= 
			\bigg(
			\overbrace{{ V \circ  
				\underbrace{
					\bigg[
						\brac{  
						\Delta(f,M,Q) \Eb + (1-\Delta(f,M,Q))\Ebpm 
					} 
					\underbrace{
						\brac{(f - {T^{\xqs}(f,M)}\big|_{E}) \circ \Phi^{-1}\big|_{\R \times \set{0}} }
					}_{\text{straightening local data}}
				\bigg]
				}_{\text{one-dimensional extension}}   }}^{\text{vertical extension}}
			\bigg) 
			\circ \Phi\,.
		\end{split}
		\label{eq.EQ}
		\end{equation}

		We now analyze the validity of {(i)-(iv)}. We break down the argument into four claims.
		
		\begin{claim}\label{claim.3.1}
			{\rm(i)} holds.
		\end{claim}
	
		\begin{proof}[Proof of Claim \ref{claim.3.1}]
			Suppose $ {T^{\xqs}(f,M)}  \equiv 0$. Then $ \Delta(f,M,Q) = 0 $ (see \eqref{eq.dfmq}). Formula \eqref{eq.EQ} simplifies to
			\begin{equation}
			\E_Q(f,M)= (1-\psi) \cdot \brac{V \circ \Eb \brac{f \circ \Phi^{-1}\big|_{\R \times \set{0}} }  }\circ \Phi\,.
			\label{eq.EQ-small}
			\end{equation}
			Since $ \Eb $ preserve nonnegativity and $ 0 \leq \psi \leq 1 $, we have $ \E_Q(f,M) \geq 0 $.
			
			\newcommand{\pqs}{{T^{\xqs}(f,M)}}
			
			Suppose $ {T^{\xqs}(f,M)} $ is not the zero polynomial. 
			
			Formula \eqref{eq.EQ} now reads as follows.
			\begin{equation}
			\E_Q(f,M) = \pqs + (1-\psi )\cdot \brac{  V\circ \Ebpm \brac{(f - \pqs\big|_E) \circ \Phi^{-1}\big|_{\R\times \set{0}}} \circ \Phi }.
			\label{eq.EQ-big}
			\end{equation}

			By Lemma \ref{lem.Txq}, we know that 
			\begin{equation}
			\pqs \in \Gk(\xqs,4k_0,CM).
			\label{eq.3.8.1}
			\end{equation}
			In particular, $ \Gk(\xqs,4k_0,CM) \neq \void $. Thanks to Lemma \ref{lem.FP-G}, there exists $ F \in \ctp(\rt) $ such that $ F = f $ on $ E $, $\norm{F}_{\ct(\rt)} \leq CM $, and
			\begin{equation}
			\jet_{\xqs}F \in \G(\xqs,E,CM) \subset \Gk(\xqs,4k_0,CM).
			\label{eq.3.8.2}
			\end{equation}
			Thanks to Lemma \ref{lem.G-G}, \eqref{eq.3.8.1}, \eqref{eq.3.8.2}, and Taylor's theorem, we have
			\begin{equation}
			\abs{\d^\alpha (F - \pqs)(x)} \leq CM\dq^{2-\abs{\alpha}}
			\for x \in 100Q \text{ and } \abs{\alpha} \leq 2\,.
			\label{eq.F-pqs}
			\end{equation}
			Combining Lemma \ref{lem.CZ}(\ref{lem.CZ.graph}) and \eqref{eq.F-pqs}, we see that
			\begin{equation}
			\norm{(F - \pqs)\circ \Phi^{-1}\big|_{\R\times\set{0}}}_{\ct(I(Q))} \leq CM.
			\label{eq.F-pqs-2}
			\end{equation}
			Here, $ I(Q) := \set{t \in \R : \Phi^{-1}(t,0) \in 100Q} $. 
			
			Since $ \Ebpm $ is bounded by Theorem \ref{thm.bd-1dpm}, we can conclude from \eqref{eq.F-pqs-2} that
			\begin{equation*}
			\norm{ \Ebpm \brac{(f - \pqs\big|_E) \circ \Phi^{-1}\big|_{\R\times \set{0}}} }_{\ct(I(Q))} \leq CM;
			\end{equation*}
			and
			\begin{equation}
			\norm{ V\circ \Ebpm \brac{(f - \pqs\big|_E) \circ \Phi^{-1}\big|_{\R\times \set{0}}} }_{\ct(\Phi(100Q))} \leq C'M.
			\label{eq.VE}
			\end{equation}
			Write 
			\begin{equation}
			G := V\circ \Ebpm \brac{(f - \pqs\big|_E) \circ \Phi^{-1}\big|_{\R\times \set{0}}}
			\label{eq.GGG}
			\end{equation}
			Thanks to \eqref{eq.F-pqs}, we see that
			\begin{equation}
			0 \leq G(x) \leq CM\dq^2
			\text{ for each }
			x \in \Phi(E \cap 2Q).
			\label{eq.Gb-small}
			\end{equation}
			By Lemma \ref{lem.TC-small} and \eqref{eq.Gb-small}, we see that
			\begin{equation}
			\abs{\d^\alpha G(x)} \leq CM\dq^{2-\abs{\alpha}}
			\text{ for } x\in \Phi(100Q).
			\label{eq.Gbd-small}
			\end{equation}
			By Lemma \ref{lem.CZ}(\ref{lem.CZ.graph}), Lemma \ref{lem.comp}, and \eqref{eq.Gbd-small}, we have
			\begin{equation*}
			\norm{G\circ \Phi}_{\ct(100Q)} \leq CM.
			\end{equation*}
			Now, thanks to \eqref{eq.Gb-small},
			\begin{equation*}
			0 \leq G\circ \Phi(x) \leq CM\dq^2
			\text{ for each } x \in E \cap 2Q
			\end{equation*}
			Thanks to Lemma \ref{lem.TC-small} again, we have
			\begin{equation}
			\abs{\d^\alpha (G\circ \Phi)(x)} \leq CM\dq^{2-\abs{\alpha}}
			\for 
			x \in 100Q \text{ and } \abs{\alpha} \leq 2.
			\label{eq.GPhi-d-small}
			\end{equation}
			
			Now, thanks to \eqref{eq.GPhi-d-small} and the fact that $ 0 \leq \psi \leq 1 $, as long as $ \pqs \geq B'M\dq^2 $ on $ 100Q $ for some sufficiently large number $ B' $, we can conclude that $ \E_Q(f,M) \geq 0 $ on $ 100Q $. 
			
			We examine the value of $ \pqs $ on $ 100Q $. 
			
			Since $ \pqs $ is not the zero polynomial, Scenario 2 in the proof of Lemma \ref{lem.Txq} must be false. Therefore, we must be in Scenario 1, namely, 
			\begin{equation}
			f(x) \geq \ub M\dq^2
			\text{ for all } x \in E \cap 2Q\,.
			\label{eq.fbig}
			\end{equation}

			First we want to show that 
			\begin{equation}
			\pqs(\xqs) \geq C(\sqrt{\ub} - 1)^2M\dq^2.
			\label{eq.pqs-big}
			\end{equation}

			Suppose for a contradiction, that $ \pqs(\xqs) \leq B_0M\dq^2 $ for some $ B_0 > 0 $ to be determined. Since $ \pqs \in \Gk(\xqs,4k_0,CM) $, for any $ x \in E \cap 2Q $, there exists $ F \in \ctp(\rt) $ with $ F(x) = f(x) $ and $ \jet_{\xqs}F = \pqs $. By Lemma \ref{lem.TC} and Taylor's theorem, we have
			\begin{equation*}
			\abs{\grad F(x)} \leq \abs{\nabla F(\xqs)} + C\norm{F}_{\ct(\rt)} \dq \leq C'(\sqrt{B_0} + 1) M\dq
			\for
			x \in 2Q.
			\end{equation*}
			By Taylor's theorem again, we have
			\begin{equation}
			F(x) \leq F(\xqs) + C\dq\cdot \sup_{y \in 2Q}\abs{\nabla F(y)} \leq C_0(\sqrt{B_0} + 1)^2M\dq^2
			\for
			x \in 2Q.
			\label{eq.F-small}
			\end{equation}
			If $ B_0 < C_0^{-1}(\sqrt{\ub} - 1) $ with $ C_0 $ as in \eqref{eq.F-small}, \eqref{eq.F-small} would contradict \eqref{eq.fbig}. Therefore, \eqref{eq.pqs-big} holds.
			
			Now, thanks to Lemma \ref{lem.TC} and \eqref{eq.pqs-big}, we have
			\begin{equation*}
			\dist{\xqs}{\set{\pqs =0}} \geq C(\sqrt{\ub} -1 )\dq.
			\end{equation*}
			For sufficiently large $ \ub $, this implies that 
			\begin{equation}
			\pqs(x) \geq C(\sqrt{\ub}-1)\dq^2
			\for
			x \in 100Q.
			\label{eq.pqs-big-Q}
			\end{equation}
			Combining \eqref{eq.psi}, \eqref{eq.EQ-big}, \eqref{eq.GGG}, \eqref{eq.GPhi-d-small}, and \eqref{eq.pqs-big-Q} and picking $ \ub $ to be sufficiently large, we see that (i) holds. This proves Claim \ref{claim.3.1}.
			
		\end{proof}

		\begin{claim}\label{claim.3.2}
			{\rm(ii)} holds.
		\end{claim}
	
		\begin{proof}[Proof of Claim \ref{claim.3.2}]
			Since the support of $ \psi $ is disjoint from the set $ E $ by \eqref{eq.psi}, and $ \widetilde{\mathcal{E}}_Q(f,M) $ in \eqref{eq.EQ} is a local extension of $ f - T^{\xqs}(f,M)\big|_{E} $, Claim \ref{claim.3.2} follows. 
		\end{proof}

		\begin{claim}\label{claim.3.3}
			{\rm(iii)} holds.
		\end{claim}
		
		\begin{proof}[Proof of Claim \ref{claim.3.3}]
			\newcommand{\pqs}{{T^{\xqs}(f,M)}}
			Suppose $ \pqs \equiv 0 \in \Gk(\xqs,4k_0,CM) $. Then formula \eqref{eq.EQ} is simplified to \eqref{eq.EQ-small}.
			
			By the definition of $ \Gk $, for each $ x \in E \cap 2Q $, there exists $ F^x \in \ctp(\rt) $ with $ F^x(x) = f(x) $ and $ \jet_{\xqs}F^x = \pqs \equiv  0 $. By Taylor's theorem, we have
			\begin{equation}
			0 \leq f(x) = F^x(x) \leq CM\dq^2
			\for
			x \in E \cap 2Q.
			\label{eq.f-small}
			\end{equation}

			Since $ \Gk(\xqs,4k_0,CM) \neq \void $, Lemma \ref{lem.FP-G} implies that there exists $ F \in \ctp(\rt) $ with $ F|_E = f $ and $ \norm{F}_{\ct(\rt)} \leq CM $. 
			
			By Lemma \ref{lem.TC-small} and \eqref{eq.f-small}, we have
			\begin{equation}
			\abs{\d^\alpha F(x)} \leq CM\dq^{2-\abs{\alpha}}
			\for x \in 100Q\,.
			\label{eq.F-d-small}
			\end{equation}
			By Lemma \ref{lem.CZ}(\ref{lem.CZ.graph}) and \eqref{eq.F-d-small}, we see that
			\begin{equation*}
			\norm{F \circ \Phi^{-1}\big|_{\R\times\set{0}}}_{\ct(I(Q))} \leq CM.
			\end{equation*}
			Here, $ I(Q):= \set{t \in \R : \Phi^{-1}(t,0) \in 100Q} $.
			
			Since $ \Eb $ is bounded, we have
			\begin{equation*}
			\norm{\Eb \brac{f \circ \Phi^{-1}\big|_{\R \times \set{0}}}}_{\ct(I(Q))} \leq CM;
			\end{equation*}
			and
			\begin{equation*}
			\norm{V\circ \Eb \brac{f \circ \Phi^{-1}\big|_{\R \times \set{0}}}}_{\ct(\Phi(100Q))} \leq C'M.
			\end{equation*}
			By \eqref{eq.f-small}, we have
			\begin{equation}
			0 \leq f \circ \Phi^{-1} \leq CM\dq^2
			\for x \in \Phi(E \cap 2Q).
			\label{eq.f-small-2}
			\end{equation}
			Set
			\begin{equation*}
			H :=  V\circ \Eb \brac{f \circ \Phi^{-1}\big|_{\R \times \set{0}}}.
			\end{equation*}
			By Lemma \ref{lem.TC-small} and \eqref{eq.f-small-2}, we have
			\begin{equation}
			\abs{\d^\alpha H(x)} \leq CM\dq^{2-\abs{\alpha}}
			\for x \in 100Q
			\text{ and } \abs{\alpha} \leq 2.
			\label{eq.H-d-small}
			\end{equation}
			By Lemma \ref{lem.CZ}(\ref{lem.CZ.graph}) and \eqref{eq.H-d-small}, we have
			\begin{equation}
			\norm{H \circ \Phi}_{\ctp(100Q)} \leq CM\,.
			\label{eq.Hphi}
			\end{equation}
			By Lemma \ref{lem.TC-small}, \eqref{eq.f-small}, and \eqref{eq.Hphi}, we have
			\begin{equation}
			\abs{\d^\alpha (H \circ \Phi)(x)} \leq CM\dq^{2-\abs{\alpha}}
			\for x \in 100Q
			\text{ and } \abs{\alpha} \leq 2.
			\label{eq.HPhi-small}
			\end{equation}
			Using \eqref{eq.psi} and \eqref{eq.HPhi-small} to estimate \eqref{eq.EQ-small}, we can conclude that 
			\begin{equation*}
			\norm{\E_Q(f,M)}_{\ct(100Q)} \leq CM.
			\end{equation*}
			This concludes the case when $ \pqs \equiv 0 $.

			Suppose $ \pqs $ is not the zero polynomial. Then formula \eqref{eq.EQ} becomes \eqref{eq.EQ-big}.
			
			Since $ \pqs \in \Gk(\xqs,4k_0,CM) $, we have
			\begin{equation}
			\norm{\pqs}_{\ct(100Q)} \leq CM.
			\label{eq.pqs-est}
			\end{equation}
			
			Using \eqref{eq.psi}, \eqref{eq.GPhi-d-small}, and \eqref{eq.pqs-est} to estimate \eqref{eq.EQ-big}, we conclude that
			\begin{equation*}
			\norm{\E_Q(f,M)}_{\ct(100Q)} \leq CM.
			\end{equation*}
			This proves the case when $ \pqs $ is not the zero polynomial.
			
			This proves Claim \ref{claim.3.3}
		\end{proof}

		\begin{claim}\label{claim.3.4}
			{\rm(iv)} holds.
		\end{claim}

		\begin{proof}[Proof of Claim \ref{claim.3.4}]
			Since $ \psi \equiv 1 $ near $ \xqs $ by \eqref{eq.psi}, we have, by Lemma \ref{lem.Txq}
			\begin{equation*}
			\jet_\xqs \E_Q(f,M) = {T^{\xqs}(f,M)}  \in \Gk(\xqs,4k_0,CM).
			\end{equation*}
			This proves Claim \ref{claim.3.4}.
		\end{proof}

		In views of Claims \ref{claim.3.1}-\ref{claim.3.4}, we see that Lemma \ref{lem.ext-loc}(A) holds.\\

		\medskip
		Now to turn to Lemma \ref{lem.ext-loc}(B). 
		
		Fix $ x \in 100Q $. 
		
		We begin by defining the set $ S_Q(x) $.

		Up to a rotation, we can assume that $ \Phi $ takes the form $ \Phi(t_1,t_2) = (t_1,t_2-\phi(t_1)) $, where $ \phi $ is as in Lemma \ref{lem.CZ}(\ref{lem.CZ.graph}). Let $ t_x \in \R $ be such that
		\begin{equation*}
		x = (t_x,\phi(t_x)). 
		\end{equation*}
		Let $ S(t_x) \subset \Phi(E \cap 2Q) $ be as in Theorems \ref{thm.bd-1d} and \ref{thm.bd-1dpm}\footnote{Here we identify $ \R\times\set{0} $ with $ \R $}. Let $ S^{\xqs} $ be as in \eqref{eq.Sx0}, with $ x_0 = \xqs $. We define
		\begin{equation*}
		S_Q(x) := \Phi^{-1}(S(t_x)) \cup \brac{S^{\xqs}\cap E}.
		\end{equation*}
		Thanks to Theorems \ref{thm.bd-1d}, \ref{thm.bd-1dpm}, and the definition of $ S^{\xqs} $ in \eqref{eq.Sx0}, we have $ \#(S(t_x)) \leq C $ and $ \#(S^{\xqs}) \leq C' $. Therefore,
		\begin{equation*}
		\#(S_Q(x)) \leq D_0,
		\end{equation*}
		for some universal constant $ D_0 $.
		
		Let $ f, g \in \ctp(E) $ with $ \norm{f}_{\ctp(E)}, \norm{f}_{\ctp(E)} \leq M $ and $ f =g $ on $ S_Q(x) $. 
		
		Next we analyze \eqref{eq.EQ}.
		
		\newcommand{\sqx}{S^{\xqs}}
		\newcommand{\txq}{T^{\xqs}}

		Since $ f = g $ on $ \sqx \cap E $ and $ \norm{f}_{\ctp(E)},\norm{g}_{\ctp(E)} \leq M $, thanks to \eqref{eq.Gg}-\eqref{eq.tx0}, we have 
		\begin{equation}
		\txq(f,M)=\txq(g,M).
		\label{eq.fg-1}
		\end{equation}
		Since $ f = g $ on $ \Phi^{-1}(S(t_x)) $, we have 
		\begin{equation}
		(f - \txq(f,M) ) \circ \Phi^{-1}= (g-\txq(g,M))\circ \Phi^{-1}
		\text{ on }
		S(t_x).
		\label{eq.fg-2}
		\end{equation}
		Thanks to \eqref{eq.fg-1}, we also have, with $ \Delta(\cdot,\cdot,\cdot) $ as in \eqref{eq.dfmq},
		\begin{equation}
		\Delta(f,M,Q) = \Delta(g,M,Q).
		\label{eq.fg-3}
		\end{equation}
		Combining \eqref{eq.fg-1}-\eqref{eq.fg-3}, we see that, for $ m = 0,1,2 $,
		\begin{equation}
		\begin{split}
		&\ddtm \left[\brac{  
			\Delta(f,M,Q) \Eb + (1-\Delta(f,M,Q))\Ebpm 
		} 
		{
			\brac{(f - \txq(f,M)|_{E}) \circ \Phi^{-1}\big|_{\R \times \set{0}} }
		}\right]\\
		&\,\,\,\,\,\,\,= \ddtm \left[\brac{  
			\Delta(g,M,Q) \Eb + (1-\Delta(g,M,Q))\Ebpm 
		} 
		{
			\brac{(g - \txq(g,M)|_{E}) \circ \Phi^{-1}\big|_{\R \times \set{0}} }
		}\right]
		\end{split}
		\label{eq.fg-4}
		\end{equation}
		Thanks to \eqref{eq.fg-1} and \eqref{eq.fg-4}, we have
		\begin{equation*}
		\d^\alpha \E_Q(f,M)(x) = \d^\alpha \E_Q(g,M)(x)
		\for\abs{\alpha} \leq 2\,.
			\end{equation*}
		This concludes the proof of Lemma \ref{lem.ext-loc}(B).
		
		Lemma \ref{lem.ext-loc} is proved.
		
	\end{proof}

	The next definition describes how we relay information to each small square in $ \Lz $ that contains no data. 
	
	\newcommand{\Lem}{\Lambda_{\rm empty}}
	
	\begin{definition}\label{def.mu}
	 	Recall $ \Lz $ and $ \Ls $ as in Definitions \ref{def.Lz} and \ref{def.Ls}. Let
	 	\begin{equation*}
	 	\Lem := \set{Q \in \Lz \setminus \Ls : \dq < 1}.
	 	\end{equation*}
	 	We define a map
	 	\begin{equation*}
	 	\mu : \Lem \to \Ls
	 	\label{eq.mu}
	 	\end{equation*}
	 	according to the following rule: Let $ Q \in \Lem $. Then $ \delta_{Q^+} \leq 1 $, but $ Q^+ \notin \Lz $. In particular, this means that there exists $ x \in 2Q^+\cap E $. By Lemma \ref{lem.CZ}(\ref{lem.CZ.cover}), there exists $ Q(x) \in \Ls $ such that $ x \in Q(x) $. We define $ \mu(Q) := Q(x) $.
	\end{definition}

	\begin{lemma}\label{lem.muQ}
		Let $ \Lem $ and $ \mu $ be as in Definition \ref{def.mu}. Then the following hold.
		\begin{enumerate}[(A)]
			\item $ 5Q \cap \mu(Q) \neq \void $ for all $ Q \in \Lem $.
			\item Let $ Q \in \Lem $. Let $ x \in Q $ and $ x' \in \mu(Q) $. Then $ \abs{x - x'} \leq C\dq $. 
		\end{enumerate}
	\end{lemma}

	\begin{proof}
		Since $ 2Q^+ \cap E \cap \mu(Q) \neq \void $ by construction, and $ 2Q^+ \subset 5Q $, (A) follows.
		
		Next we prove (B). Thanks to (A), it suffices to show that
		\begin{equation}
		\delta_{\mu(Q)} \leq C\delta_Q.
		\label{eq.dmuQ}
		\end{equation}
		Suppose toward a contradiction, that $ \delta_{\mu(Q)} \geq 100\dq $. Then $ 2Q^+ \subset 2\mu(Q) $. Since $ \mu(Q) \in \Ls \subset \Lz  $, this would contradict the fact that $ 2Q^+ \notin \Lz $. Therefore, \eqref{eq.dmuQ} holds. This proves (B). Lemma \ref{lem.muQ} is proved.
	\end{proof}

	We now have all the ingredients to prove Theorem \ref{thm.bd}.

	\begin{proof}[Proof of Theorem \ref{thm.bd}]
		
		Recall $ \Lz $, $ \Ls $, $ \Lem $, $ \mu $ as in Definitions \ref{def.Lz}, \ref{def.Ls} and \ref{def.mu}
		
		We assign a local operator to each element in $ \Lz $ according to the following.
		
		\newcommand{\xms}{{x_{\mu(Q)}^\sharp}}
		
		\newcommand{\eqs}{\E_Q^\sharp}
		
		\begin{enumerate}[Type 1]
			\item Suppose $ Q \in \Ls $, i.e., $ E \cap 2Q \neq \void $. We set $ \eqs := \E_Q $, where $ \E_Q $ is as in Lemma \ref{lem.ext-loc}. Let $ \xqs $ be as in Lemma \ref{lem.CZ}(\ref{lem.CZ.rep}).
			
			\item Suppose $ Q \in \Lem $. We set $ \eqs := T_w^{\{\xms\}} \circ T^{\xms} $, where $ \xms $ is as in Lemma \ref{lem.CZ}(\ref{lem.CZ.rep}), $ T^{\xms} $ is as in Definition \ref{def.Tx0} (with $ x_0 = \xms $), and $ T_w^{\{\xms\}} $ is the operator in Lemma \ref{lem.WT} associated with the singleton Whitney field $ \wtp\big(\{\xms\}\big) $. Let $ \xqs $ be as in Lemma \ref{lem.CZ}(\ref{lem.CZ.rep}).

			\item Suppose $ Q \in \Lz \setminus (\Ls\cup\Lem) $, namely, $ E \cap 2Q = \void $ and $ \dq = 1 $. We set $ \eqs :\equiv 0 $. Let $ \xqs $ be as in Lemma \ref{lem.CZ}(\ref{lem.CZ.rep}).
			
		\end{enumerate}

		Let $ \set{\theta_Q : Q \in \Lz} $ be a partition of unity subordinate to $ \Lz $, such that
		\begin{equation}
		\sum_{Q \in \Lz} \theta_Q \equiv 1
		,\,
		0 \leq \theta_Q \leq 1,\,
		\supp{\theta_Q} \subset \frac{9}{8} Q,\text{ and }
		\abs{\d^\alpha\theta_Q}\leq C\delta_Q^{2-\abs{\alpha}}
		\,\text{ for }\abs{\alpha} \leq 2\,.
		\label{eq.theta-Q}
		\end{equation}
		
		Given $ f,M $, we define
		\begin{equation}
		\E(f,M)(x) := \sum_{Q \in \Lz}\theta_Q(x) \cdot \eqs(f,M)(x).
		\label{eq.E}
		\end{equation}

		Since $ \eqs(f,M) \geq 0 $ on $ 2Q $ for each $ Q \in \Lz $, we have $ \E(f,M) \geq 0 $ on $ \rt $.
		
		Since $ \eqs(f,M) = f $ on $ E \cap 2Q $ for each $ Q \in \Lz $, we have $ \E(f,M) = f $ on $ E $.
		
		Fix $ x \in \rt $. We compute the derivatives of $ \E(f,M) $ at $ x $.
		
		\newcommand{\eqsp}{{\mathcal{E}_{Q'}^\sharp}}
		Let $ Q \in \Lz $ such that $ Q \ni x $. We can write
		\begin{equation}
		\begin{split}
		\d^\alpha \E(f,M)(x) &= \sum_{Q'\touch Q(x)}\theta_{Q'}(x)\cdot\d^\alpha \E_{Q'}(f,M)(x) \\
		&\,\,\,\,\,\,\,\,\,\,+ \sum_{\substack{Q'\touch Q\\0 < \beta \leq \alpha}} \d^\beta\theta_{Q'}(x) \cdot \d^{\alpha-\beta}\brac{ \eqs(f,M) - \eqsp (f,M) } (x)\,.
		\end{split}
		\label{eq.dEfM}
		\end{equation}

%		By our choice of $ \xqs $ and $ x_{Q'}^\sharp $, we have
%		\begin{equation*}
%		\abs{x_{Q}^\sharp - x}, \abs{x_{Q'}^\sharp - x} \leq C\delta_{Q}\,.
%		\end{equation*}

		\begin{claim}\label{claim.last}
			Fix $ x \in \rt $. Let $ Q \in \Lz $ such that $ Q \ni x $. Given $ Q' \in \Lz $ such that $ Q' \touch Q $, we have
			\begin{equation}
			\abs{\d^\alpha (\eqs(f,M) - \eqsp(f,M))(x)} \leq CM\delta_{Q}^{2-\abs{\alpha}}
			\text{ for }\abs{\alpha} \leq 2.
			\label{eq.patch-est}
			\end{equation}
		\end{claim}

		\begin{proof}[Proof of Claim \ref{claim.last}]
			
			Fix $ \alpha $ with $ \abs{\alpha} \leq 2 $. By the triangle inequality, we can write
			\begin{equation}
			\begin{split}
			\abs{\d^\alpha (\eqs(f,M) - \eqsp(f,M))(x)}
			&\leq \abs{\d^\alpha (\eqs(f,M) - \jet_{x_{Q}^\sharp}\eqs(f,M))(x)}\\
			&\,\,\,\,\,\,\,\,+ \abs{\d^\alpha (\eqsp(f,M) - \jet_{x_{Q'}^\sharp}\eqsp(f,M))(x)}\\
			&\,\,\,\,\,\,\,\,+
			\abs{\d^\alpha(\jet_{x_{Q}^\sharp}\eqs(f,M) - \jet_{x_{Q'}^\sharp}\eqsp(f,M))(x)}\\
			&=: \eta_1 + \eta_2 + \eta_3.
			\end{split}
			\label{eq.eta0}
			\end{equation}
			
			By Taylor's theorem and Lemma \ref{lem.CZ}(\ref{lem.CZ.cover}),
			\begin{equation}
			\eta_1, \eta_2 \leq CM\delta_{Q}^{2-\abs{\alpha}}.
			\label{eq.eta-12}
			\end{equation}
			
			Now we estimate $ \eta_3 $. We want to show that
			\begin{equation}
			\eta_3 \leq CM \delta_{Q}^{2-\abs{\alpha}}.
			\label{eq.eta-3}
			\end{equation}
			
			\begin{enumerate}[\text{Case }1]
				\item If either $ Q $ or $ Q' $ is of Type 3, then \eqref{eq.eta-3} follows from the triangle inequality, Lemma \ref{lem.WT} and Lemma \ref{lem.ext-loc}.

				For the rest of the cases, we assume that neither $ Q $ nor $ Q' $ is of Type 3.
				
				\item Suppose both $ Q, Q' \in \Ls $. Recall from Lemma \ref{lem.ext-loc} that $ \jet_{x_{Q}^\sharp}\eqs(f,M) \in \Gk(x_{Q}^\sharp,4k_0,CM) $ and $ \jet_{x_{Q'}^\sharp}\eqsp(f,M) \in \Gk(x_{Q'}^\sharp,4k_0,CM) $. Thus, \eqref{eq.eta-3} follows from Lemma \ref{lem.G-G} and Taylor's theorem.

				\renewcommand{\xms}{{x_{\mu(Q')}^\sharp}}
				\item Suppose one and only one of $ Q, Q' $ belongs to $ \Ls $. By symmetry, we may assume that $ Q \in \Ls $ and $ Q' \in \Lem $.  This means that $ \E_{Q'} = T_w^{\{\xms\}} \circ T^{\xms} $, with $ T^{\xms} $ as in Definition \ref{def.Tx0} and $ T_w^{\{\xms\}} $ as in Lemma \ref{lem.WT}. Thanks to Lemma \ref{lem.CZ}(\ref{lem.CZ.cover}) and Lemma \ref{lem.muQ}(B), we have
				\begin{equation}
				\eqindent
				\abs{\xqs - \xms}, \abs{x - \xqs}, \abs{x - \xms} \leq C\dq.
				\label{eq.xqs-xms}
				\end{equation}
				By the triangle inequality, we have
				\begin{equation}
				\eqindent
				\begin{split}
				\eta_3 &\leq \abs{\d^\alpha\brac{\jet_{x_{Q}^\sharp}\eqs(f,M) - T^{\xms}(f,M)}(x)}\\
				&\enskip\enskip+
				\bigg|
				\d^\alpha 
				\bigg(T^{\xms}(f,M) - T_w^{\{\xms\}} \circ T^{\xms}(f,M)\bigg)(x)
				\bigg|\\
				&= \eta_3^{(1)} + \eta_3^{(2)}.
				\end{split}
				\label{eq.eta-3-1}
				\end{equation}
				
				By Taylor's theorem and \eqref{eq.xqs-xms}, we have
				\begin{equation}
				\eqindent
				\eta_3^{(2)} \leq CM\dq^{2-\abs{\alpha}}.
				\label{eq.e32}
				\end{equation}
				
				To estimate $ \eta_3^{(1)} $, we write
				\begin{equation}
				\eqindent
				\begin{split}
				\eta_3^{(1)}
				&=
				\abs{\d^\alpha \brac{ T^{\xqs}(f,M) - T^{\xms}(f,M) }(x)}\enskip \enskip\enskip \text{(Lemma \ref{lem.ext-loc})}\\
				&\leq 
				CM\dq^{2-\abs{\alpha}}.
				\enskip \enskip\enskip
				\text{(Lemma \ref{lem.G-G}, Taylor's theorem, and \eqref{eq.xqs-xms})}
				\end{split}
				\label{eq.e31}
				\end{equation}
				
				Thus, \eqref{eq.eta-3} follows from \eqref{eq.eta-3-1}, \eqref{eq.e32}, and \eqref{eq.e31}.

				\renewcommand{\xms}{{x_{\mu(Q)}^\sharp}}
				\newcommand{\xmsp}{{x_{\mu(Q')}^\sharp}}
				\renewcommand{\eqsp}{{\mathcal{E}_{Q'}^\sharp}}
				
				\item Suppose $ Q, Q' \in \Lem $. By construction, we have $ \eqs = T_w^{\{\xms\}} \circ T^{\xms} $ and $ \eqsp = T_w^{\{\xmsp\}}\circ T^{\xmsp} $, with $ T^{\{\cdot\}} $ as in Definition \ref{def.Tx0} and $ T_w^{\{\cdot\}} $ as in Lemma \ref{lem.WT}. 
				
				Thanks to Lemma \ref{lem.CZ}(\ref{lem.CZ.cover}) and Lemma \ref{lem.muQ}(B), we have
				\begin{equation}
				\eqindent
				\delta_{\mu(Q)}, \delta_{\mu(Q')}, \abs{x - \xms}, \abs{x - \xmsp}, \abs{x - \xmsp} \leq C\dq.
				\label{eq.dmq}
				\end{equation}

				We note that
				\begin{equation*}
				\eqindent
				\begin{split}
				\jet_{\xms} \brac{T_w^{\{\xms\}} \circ T^{\xms}(f,M)} &= T^{\xms}(f,M), \text{ and }\\
				\jet_{\xmsp} \brac{T_w^{\{\xmsp\}}  \circ T^{\xmsp}(f,M)} &= T^{\xmsp}(f,M).
				\end{split} 
				\end{equation*}
				Therefore, by the triangle inequality, we can write
				\begin{equation}
				\eqindent
				\begin{split}
				\eta_3 &\leq \abs{\d^\alpha\brac{ T_w^{\{\xms\}} \circ T^{\xms}(f,M) - T^{\xms}(f,M) }(x) }\\
				&\enskip\enskip+ 
				\abs{\d^\alpha\brac{ T_w^{\{\xmsp\}} \circ T^{\xmsp}(f,M) - T^{\xmsp}(f,M) }(x) }\\
				&\enskip\enskip+ \abs{\d^\alpha \brac{ T^{\xms}(f,M)   - T^{\xmsp}(f,M)   }(x)}\\
				&=: \eta_3^{(1)} + \eta_3^{(2)} + \eta_3^{(3)}.  
				\end{split}
				\label{eq.eta-3-2}
				\end{equation}
				By Taylor's theorem and \eqref{eq.dmq}, we have
				\begin{equation}
				\eqindent
				\eta_3^{(1)}  + \eta_3^{(2)} \leq CM\brac{\abs{\xms - x} + \abs{\xmsp - x}}^{2-\abs{\alpha}} \leq CM\dq^{2-\abs{\alpha}}.
				\label{eq.ee}
				\end{equation} 
				By Lemma \ref{lem.G-G}, Taylor's theorem, and \eqref{eq.dmq}, we have
				\begin{equation}
				\eqindent
				\eta_3^{(3)} \leq CM\delta_{Q}^{2-\abs{\alpha}}.
				\label{eq.eee}
				\end{equation}
				Thus, \eqref{eq.eta-3} follows from \eqref{eq.eta-3-2}, \eqref{eq.ee}, and \eqref{eq.eee}.

			\end{enumerate}
			We have now finished analyzing all the cases. 
			
			Thus, \eqref{eq.patch-est} follows from \eqref{eq.eta0}, \eqref{eq.eta-12}, and \eqref{eq.eta-3}. Claim \ref{claim.last} is proved.

		\end{proof}

		Now, using \eqref{eq.bip} and Lemma \ref{lem.ext-loc}(A) to estimate the first sum in \eqref{eq.dEfM}, and using \eqref{eq.bip}, \eqref{eq.theta-Q} and \eqref{eq.patch-est} to estimate the second sum in \eqref{eq.dEfM}, we can conclude that
		\begin{equation*}
		\norm{\E(f,M)}_{\ct(\rt)} \leq CM\,.
		\end{equation*}

		This proves part (A) of Theorem \ref{thm.bd}.

		Now we turn to part (B). 
		
		Fix $ x \in \rt $. Let $ Q(x) \in \Lz $ such that $ Q(x) \ni x $. We define
		\begin{equation*}
		S(x) := \brac{\bigcup_{\substack{Q'\touch Q(x)\\Q' \in \Ls}} S_{Q'}(x) } \cup \brac{\bigcup_{\substack{Q'\touch Q(x)\\ Q' \in \Lem}}S_{\mu(Q')(x)}},
		\end{equation*}
		with $ S_Q(x), S_{\mu(Q')}(x) $ as in Lemma \ref{lem.ext-loc}(B), and $ \mu $ as in Definition \ref{def.mu}.
		
		By \eqref{eq.bip} and Lemma \ref{lem.ext-loc}(B), we have
		\begin{equation*}
		\#(S(x)) \leq D,
		\end{equation*}
		for some universal constant $ D $.
		
		Let $ f , g \in \ctp(E) $ with $ \norm{f}_{\ctp(E)}, \norm{g}_{\ctp(E)} \leq M $ and $ f = g $ on $ S(x) $. By the construction of $ S(x) $, we see that 
		\begin{equation}
		\d^\alpha \eqsp(f,M)(x) = \d^\alpha \eqsp(g,M)(x)
		\for\abs{\alpha} \leq 2,
		\label{eq.EQfg}
		\end{equation}
		for all $ Q' \in \Lz $ such that $ \frac{9}{8}Q' \ni x $. From \eqref{eq.theta-Q}, \eqref{eq.dEfM}, and \eqref{eq.EQfg}, we see that
		\begin{equation*}
		\d^\alpha \E(f,M)(x) = \d^\alpha \E(g,M)(x)
		\for \abs{\alpha} \leq 2.
		\end{equation*}
		This proves part (B). Theorem \ref{thm.bd} is proved.
	\end{proof}

	%\bibliographystyle{plain}
	%\bibliography{Whitney-bib}

\end{document}